\documentclass[12pt]{article}

\usepackage{amsmath,     %
            amssymb,     %
            amsthm,      %
            enumerate,   %
            epsfig,      %
            graphics,    %
            indentfirst, %
           makeidx
  }    %

\usepackage[T2A]{fontenc}
\usepackage{amsmath}
\usepackage{xcolor}

\newtheorem{theorem}{Theorem}[section]          %
\newtheorem{lemma}{Lemma}[section]              %
\newtheorem{remark}{Remark}[section]            %
\def\OL#1{\overline {#1}}

\def\UL#1{\underline{#1}}

\let\wt\widetilde

\let\bo\partial

\def\N{\newline}

\def\7{\mathaccent"7017}

\def\q{\quad}

\def\Q{\qquad}

\newcommand{\bi}[1]{\mbox{\boldmath $ {#1} $}} 

\newcommand{\mbb} {\mathbb}

\newcommand{\mc}[1]{\mbox{$\mathcal {#1}$}}

\newcommand{\mr} {\mathrm}

\newcommand {\mbf} {\mathbf}

\newcommand{\vertiii}[1]{{\left\vert\kern-0.25ex\left\vert\kern-0.25ex\left\vert #1
    \right\vert\kern-0.25ex\right\vert\kern-0.25ex\right\vert}}

\begin{document}

\title{On the error control  at numerical solution\\  of reaction-difusion equations
\footnote{ The paper is supported the  RFBR N 15-01-08847а.}
     }
   \author{  V. Korneev}

 \maketitle

\begin{center}
\small{ St. Petersburg State  University,
         Russia } \\
 \vskip 3mm
  \footnotesize{\tt Vad.Korneev2011@yandex.ru
  }
 \end{center}

 \begin{abstract}
 We suggest guaranteed, robust a posteriori error bounds for approximate solutions of the reaction-diffusion equations, modeled by the equation $-\Delta u+\sigma u= f$ in $\Omega$ with any $\sigma={\mathrm{const}}\ge 0$. We also term our bounds consistent due to one specific property. It assumes that their orders of accuracy in respect to  mesh size $h$ are the same with the respective not improvable in the order a priori bounds. Additionally, it assumes that the pointed out equality of the orders is provided by the testing flaxes not subjected to equilibration. For any $\sigma\in [0,\sigma_*]$, the rirght part of the new general bound of the paper contains, besides the usual diffusion term, the $L_2$ norm of the residual with the factor $1/\sqrt{\sigma_*}$, where $\sigma_*$ is some critical value. For the solutions by the finite element method, it is estimated as $\sigma_*\ge ch^{-2},\,\,c={\mathrm{const}}$, if $\partial \Omega$ is sufficiently smooth and the finite element space is of the 1$^{\mathrm{st}}$ order of accuracy at least. In general, at the derivation of a posteriori bounds, consistency is achieved by taking adequately into account the difference of the orders of the $L_2$ and $H^1$ error norms, that can be done in various ways with accordingly introduced $\sigma_*$. Two advantages of the obtained consistent posteriori error bounds deserve attention.  They are better accuracy and the possibility to avoid the use of the equilibration in the flax recovery procedures, that may greatly simplify these procedures and make them much more universal. The technique of obtaining the consistent a posteriori bounds was briefly exposed by the author in [{\em arXiv:1702.00433v1 [math.NA]}, 1 Feb 2017] and [$Doklady Mathematics$, ${\mathbf{96}}$ (1), 2017, 380-383].
 \end{abstract}

 \section{\normalsize Introduction}\label{Se:In}
 \setcounter{equation}{0}

 \par For the successful  error control of approximate solutions to the boundary value problems, the guaranteed a posteriori error majorant must be sufficiently accurate and cheap in a sense of the computational work.  In particular,  it is natural to expect that
 the computational cost of an  error estimate does not exceed the cost of the numerical solution of the problem.
 For the elliptic equations of the second order, most often the  FEM's (finite element methods) of the  class $C$ are used with the solutions belonging to finite-dimensional subspaces of $C(\Omega)\cap H^1(\Omega)$. The second derivatives of such approximate solutions, which are needed to calculate the residuals, are not defined. Therefore, the majorants of the error are calculated by using more smooth
 testing flows, which are found with the help of FEM flows by means  of special procedures, termed \textit{flax recovery procedures}.  They significantly influence  the accuracy of a posteriori error bounds.
\par  The need to improve the smoothness of numerical flows  without losses in  accuracy motivated the development of the flax recovery procedures, which currently are  numerous.
 There are several types of such procedures attempting as the reduction of the
  right parts of the error bounds, so
 producing the flaxes, equilibrated in a weak or a strong form,  and, for the result, allowing to reduce the coefficients before the residuals or even to remove them from the bounds.  For making  the equilibration simpler, it is usually done  not for the right part $f$ of the elliptic equation, but for its approximation $\hat{f}$, which on every finite element is defined, {\em e.g.},  as the orthogonal $L_2$-projection  on  the space of the finite element.
Considerable attention was given to procedures in which the evaluation of the smoothed and equilibrated flaxes was reduced to the resolving of auxiliary local problems of approximation/interpolation.
    Among great number of works devoted to the mentioned topics, we are able to refer to a few, in particular, to
  Zienkiewicz and Zhu \cite{ZienkiewiczZhu:92}, Ainsworth and Oden \cite{AinsworthOden:00}, Babuska \textit{et al.} \cite{BabuskaStrouboulis:01, BabuskaWhitemanStrouboulis:11}, Ern \textit{et al.} \cite{ErnStephansenVohralik:2009}, Cheddadi \textit{et al.}  \cite{CheddadiRadekFucikPrietoVohralik:2009}, Kai and Zhang \cite{Cai and S. Zhang :2010} and Ainsworth and Vejchodsky \cite{AinsworthVejchodsky:2015}, where additional extensife bibliography  can be found.

  \par At the creation of  flax recovery procedures, three requirements are paid attention :
   \par ${\bi \alpha}$) preserving orders of accuracy, \textit{e.g.}, in the  energy and other norms,  the same as the numerical flaxes determined by the approximate solution of the boundary value problem,
\par ${\bi \beta}$) obtaining the balanced or weakly balanced recovered flaxes, which element wise provide  smallness of the residual type terms  or make them equal to zero,
\par ${\bi\gamma}$) providing linear or almost linear computational cost.
 \N Obviously, the requirement  ${\bi\beta}$) complicates the procedures, makes them  directly dependent on the specific boundary value problem to be solved and, therefore, less universal.  In a number of  the a posteriori  error majorants  the orders of smallness of the   residual type component and of the others, evidently, differ.  The implementation of ${\bi\beta}$) gives the possibility to reduce the order of the first  component and make it comparable  with the error of the approximate solution. A consistent and efficient implementation of this approach is found   in  Ainsworth and Oden \cite{AinsworthOden:00}  and        Ainsworth and Vejchodsky
\cite{AinsworthVejchodsky:2015,  AinsworthVejchodsky:2011}.

\par In the present work, we explore the possibility of obtaining the  a posteriori  error majorants all terms of which would have the same order of smallness  with respect to the mesh size $h$ independently of the equilibration. In this relation,  we use the term
\textit{consistent a posteriori error estimate/majorant}  attributed to
  the a posteriori estimates/majorants, possessing the following property:
  on the test flaxes, that meet the requirements  ${\bi \alpha}$), ${\bi \gamma}$), they provide
the same  orders of accuracy  with the corresponding a priori
 error estimates.
Obviously, consistent majorant is unimprovable
in the  order of accuracy, if the  a priori error estimate
of the numerical method  is unimprovable in the same sense.\footnote{In the literature,  the term "consistency"\!\, is
often understood in a weaker sense  as turning  majorant in zero upon  substitution of the exact solution in it.}

\par The typical summand of the majorants of the error energy norms is  $\theta^{1/2}\|\hat{f}-\sigma u_h+ {\bf z}\|_{L_2(\Omega)}$, where $\sigma$ is the nonnegative constant reaction coefficient in the equation of reaction-diffusion,  $u_h$ is the approximate solution, and ${\bf z}$ is the testing flax vector function.  Different authors come to  different expressions for $\theta$ and, in particular, to
 {\em a})   $\theta=1/\sqrt{\sigma}$ for all $\sigma > 0$,   {\em b})  $\theta=$const  for $\sigma \equiv 0$,   {\em c})  $\theta=(ch)^2$  for $\sigma \le (ch)^{-2}$,
and  {\em d}) $\theta=0$  for $\sigma \le (ch)^{-2}$.  Several examples of majorants of these types are given in  the next section. It is not difficult to come to the conclusion that in the cases {\em a}) and {\em b}) the a posteriory bounds can be larger  in $h^{-1}$  times,  and more in the case  {\em a}),   than the energy norm of the error, if the test flaxes  satisfy only the requirements  ${\bi\alpha}$) and ${\bi\gamma}$).  This is a consequence of the fact that these bounds are nonconsistent.
The orders of  smallness of a posteriori majorants, related to {\em c}) and {\em d}), are equal to the orders provided by the corresponding a priori trror bounds. At least this is true for the methods with the linear thetrahedral finite elements and the problems with the exact solutions belonging to $H^2(\Omega)$.  However, in
 a strict sense, the aposteriory bounds
related to {\em c}) and {\em d}) can be also inconsistent when they are based,  as it  happens most often, on the use of the equilibrated testing flaxes.  In other words, they can produce not satisfactory bounds, if testing flaxes satisfy ${\bi\alpha}$) and ${\bi\gamma}$), but are not equilibrated.

  \par There exists
  other option to improve the accuracy of a posteriori error majorants and  at the same time to simplify the procedures of flax recovery.
  Such an option is provided by the consistent a posteriori error bounds which are derived in this paper,
   following techniques, briefly presented in \cite{Korneev:2016, Korneev:2017-1, Korneev:2017-2}.
  The above mentioned $L_2$-norms of the residual type enter right parts of these bounds with the  multiplier  $\theta^{1/2}=ch $, as in the case {\em c}), but without the assumption of   the testing flax equilibration.  Therefore,  not only  the accuracy of the a posteriori bounds is improved, but  simultaneously the flax recovery procedures can be noticeably simplified.

  \par It is worth noting that the structure of modern a posteriori error bounds for approximations of the solutions of the reaction-difusion equations
   is met in the works of Aubin, see \cite{Aubin:72}. His majorant corresponds to the case {\em a}), does not assume the equilibration and
    is accurate for sufficiently big values of $\sigma$. However, its accuracy drops when
 $\sigma\rightarrow 0$ and at $\sigma = 0$ it losses it sense. A number of later majorants, having similar structures,
 resulted from the attempts to improve accuracy by different remedies. In order to come to our bounds, we use the  new way of  their derivation, which adequately  takes  into account the difference in the orders of   $L_2$  and $H^1$  error norms  of   approximate solutions.
The majorants proposed in this work are  defined for all $\sigma \ge 0$, for  $\sigma \ge (ch)^{-2}$ coincide with the majorant of Aubin, do not lose precision for $ \sigma \in[ 0, ch^{-2}],\,\,c ={\mr{const}}$, and are consistent.

\par  Simplicity of evaluation of  constants in a posteriory error majorants is very important for the practice.  In the paper, we suggest   consistent majorants of two types with differently defined constants. In one of them, the constants depend on the constants in local
bounds of approximation in $L_2(\tau_r)$ and stability in $H^1(\delta^{(r)})$ of the quasiinterpolation operator $H^1(\delta^{(r)})\rightarrow {\mbb V}_h(\tau_r)$, where ${\mbb V}_h(\tau_r)$ is the space induced by the finite element $\tau_r$, $\delta^{(r)}$ is the smallest patch of
the finite elements neighbouring $\tau_r$. In particular, the quasiinterpolation operator of Scott and Zhang \cite{ScottZhang:90}  can be used. As was mentioned, consistency with  a priori error estimates, unimprovable in the order of accuracy, implies that a posteriori error majorant is accurate in the same sense. The exactness in the order of the majorant can be also  confirmed by the inverse
 bounds and by the so called bounds of local effectiveness. Due to the discussed above properties of  our majorants, they are majorated by some known majorants and,  as a consequence, some known inverse bounds can be easily adapted to our majorants. We consider an example of such bound.

\par The paper is organized as follows. Section~\ref{Se:Examples}
 contains the
formulation of the  boundary value problem of reaction-diffusion, examples of known error majorants, similar in structure to the ones suggested in the paper, and their brief discussion from the point of view of consistency.
 In Section~\ref{Se:consist}, we propose the new general a posteriori majorant for the error of
 approximation of the exact solution to the boundary value problem by an arbitrary sufficiently smooth function $v$ that satisfies the essential  boundary conditions.
  It is  defined for $\forall \sigma \ge 0$  and coincides
  with the majorant of Aubin \cite{Aubin:72}
  in the case of $\sigma$,
 exceeding a certain critical value $\sigma_*$. Therefore, it can be considered as the improved Aubin's majorant.
   Several versions of the  majorant  are discussed, related to
different ways of defining $\sigma_*$ and, respectively, coefficients of the majorant. The easiest one corresponds to the majorant applicable to the approximate solutions $v$ by Galerkin method with coordinate functions belonging to the space $ H^2(\Omega)$. This means that
it is directly applicable in the isogeometric analysis, see Cottrel etc. \cite{CottrellHughesBazilevs:2009}, which makes use of coordinate functions of a higher smoothness. Majorants of Section~~\ref{Se:consist}  are quite general, they do not address  properties of the mesh methods.

\par Consistent a posteriori error estimates for  solutions  by the finite element method   are presented in Sections~\ref{Se:FEM} and \ref{Se:cons-inverse}. Theorems~\ref{Th: Kor} and \ref{Th: Kor+1},  proved in Section~\ref{Se:FEM} suggest different approaches  to evaluation of constants.
  In the majorant of Theorem~\ref{Th: Kor+1} they are expressed through the constants in the estimates of approximation by means of the quasiinterpolation projection operator of Scott and Zhang \cite{ScottZhang:90}.  In Section~\ref{Se:cons-inverse},  some properties of new  a posteriori estimators are discussed,  their consistency with the  known  unimprovable a priori error estimates  is shown and supported  by the derivation of an  inverse like bound.

\par In what follows $\left\| \phi \right\| _{H^{k} ({\mathcal G})} $ is the norm in the Sobolev space $H^{k} ( {\mathcal G})$ on the domain $ {\mathcal G}$
\[\left\| \phi \right\| _{H^{k} ({\mathcal G} )}^{2} =\left\| \phi \right\| _{L_{2} ({\mathcal G})}^{2} +\hspace*{-1 mm} \sum _{l=1}^{k} |\phi |_{H^{l} ({\mathcal G} )}^{2},\q
|\phi |_{H^{l} ({\mathcal G})}^{2}= \hspace*{-1mm} \sum _{|q|=l} \int _{{\mathcal G} }( \frac{\partial ^{l} \phi}{
\partial x_1^{q_1 } \partial x_2^{q_2} \dots \partial x_m^{q_m }} )^{2} dx \, ,
\]
where
$\left\| \phi \right\| _{L_{2} ({\mathcal G})}^{2} =\int _{{\mathcal G} }\phi ^{2} dx$ and $q=(q_1,q_2,\dots, q_m),\,\, q_k\ge 0,\,\, |q|=q_1+q_2+\dots+ q_m$.
If ${\mathcal G}=\Omega$,  for $\|\cdot \|_{L_2 (\Omega)},\,\, \| \cdot\| _{H^{k} (\Omega )}$ and
$| \cdot| _{H^{k} (\Omega )}$ will also be used simpler symbols
$\| \cdot \|_0,\,\,\|\cdot \|_k$ and $|\cdot |_k$, respectively.  If on the entire boundary $\bo \Omega $ or on its part  $\Gamma_D $
the homogeneous Dirichlet boundary  condition is specified, then for the corresponding subspaces of functions from $ H^1(\Omega)$
we use the notations
  $ \7H^1(\Omega):=\{v\in H^1(\Omega):v|_{\bo \Omega}=0 \}$ and $ \7H^1_{\Gamma_D}(\Omega):=\{v\in H^1(\Omega):v|_{\Gamma_D}=0 \}$. We need also the spaces $\7{H}^1(\Omega, \Delta)=\{v\in \7{H}^1(\Omega): \Delta v\in L_2(\Omega)\} $ and  $\7{H}_{\Gamma_D}^1(\Omega, \Delta)=\{v\in \7{H}_{\Gamma_D}^1(\Omega): \Delta v\in L_2(\Omega)\} $.

  \par The  finite element space will be denoted ${\mbb V}_h(\Omega)$ and by definition  $\7{\mbb V}_h(\Omega)=\{ v\in {\mbb V}_h(\Omega): v|_{\bo \Omega}=0 \}$.

\par Everywhere below it is assumed that on $\Omega \subset \mathbb{R}^m,\,\, m=2,3$, the assemblage of compatible and, generally speaking, curvilinear finite elements is given with each finite element occupying  domain
 $\tau_r,\,\, r=1,2,\dots,{\mc R}$. Sometimes, we use the notation ${\mc R}$ also for the set of numbers of finite elements. The finite elements are defined by  sufficiently smooth mappings $x={\mc X}^{(r)}(\xi):\tau_{_\vartriangle}\rightarrow \tau_r$ of the reference  element, defined on the standard triangle or tetrahedron $\tau_{_\vartriangle}$. The span of the coordinate functions of the reference element is the space ${\mc P}_p$ of polynomials of degree $p \in {\mbb N}^+ $.
   If $p>1$, sometimes we use also the notation
  ${\mbb V}_h(\Omega)={\mbb V}_{h,p}(\Omega)$. If  other is not mentioned, it is always assumed that the finite element assemblage satisfies the  \textit{generalized conditions of quasiuniformity} with the mesh parameter $h>0$, which
  can be understood as the maximum of diameters of finite elements.
The generalized conditions of quasiuniformity for the mappings, defining finite elements   (and finite element mesh), as well as  the \textit{generalized shape quasiuniformity  (reqularity)  conditions}  for curvilinear finite elements can be found, for instance,
in Korneev and Langer \cite[Section 3.2]{KorneevLanger:2015}.

\par As a rule in applications   ${\mbb V}_h(\Omega) \subset C(\Omega)\cap H^1(\Omega)$.
At the same time,  in the isogeometric analysis,  more smooth finite dimensional spaces
${\mbb V}_h(\Omega)={\mbb V}_h^1(\Omega) \!\subset\! C^1(\Omega)\!\cap\! H^2(\Omega)$, see Kortell  {\em et al.} \cite{CottrellHughesBazilevs:2009}, are in the use in computational schemes for solving elliptic equations of 2$^{\mr{nd}}$ order. Superscript $l$
 in the notations ${\mbb V}_h^l(\Omega),\,\,{\mbb V}_{h,p}^l(\Omega)$ assumes  inclusion of these spsces in $
 C^l(\Omega)\!\cap\! H^{l+1}(\Omega)$.

\section{\normalsize Model problem, examples of a posteriori error majorants}  \label{Se:Examples}

\par One of the earliest is the a posteriori error majorant of Aubin \cite{Aubin:72}. We illustrate it on the model problem
\begin{equation}\label{poiss}
\begin{array}{c}
{\mathfrak L} u \equiv - {\mr{div }} ({\bf A} {\mr{grad}} \, u) + \sigma u=F(x),\q x=(x_1,x_2,\dots,x_m)\in\Omega \subset {\mbb R}^m\,, \\
u\big|_{\Gamma_D}=\psi_D\,, \q
- {\bf A} \nabla \,u\cdot  {\bi \nu} \big|_{\Gamma_N}=\psi_N\,,
\end{array}
\end{equation}
where $\Gamma_D,\,\,\Gamma_N$ are disjoint,
for simplicity, simply connected
parts of the boundary
 $ \bo \Omega=\OL{\Gamma}_D\cup \OL{\Gamma}_N,\,\,{\mr{mes}}\, \Gamma_D>0$, ${\bi \nu}$ is the internal normal to $ \bo \Omega$, ${\bf A} $ --
symmetric
$m\times m$ matrix that satisfies the inequalities
\begin{equation}\label{A}
\mu_1 {\bi \xi}\cdot  {\bi \xi}\le {\bf A} {\bi \xi}\cdot {\bi \xi}\le \mu_2 {\bi \xi}\cdot  {\bi \xi}\,, \Q 0<\mu_1,\mu_2={\mr{const}},
\end{equation}
for any $x\in\Omega$ and ${\bi \xi}\in {\mbb R}^m$.
The reaction coefficient $\sigma \ge 0$ is assumed to be constant and, in some cases, element wise constant. The boundary of $\Omega$, the coefficients of the matrix ${\bf A} $, and right part $f$ are always considered as sufficiently smooth, in particular, $f\in L_2(\Omega)$, if the requirements on the  smoothness are not formulated differently.
\par Our  primal interest will be the error estimates in the energy norm
\begin{equation}\label{E-norm}
|\!\!\,\! |\!\!\,\!| u|\!\,\!\!|\!\,\!\!|=\left(\|u\|_{\bf A}^2+\sigma \|u\|_{L_2(\Omega)}^2\right)^{1/2}\,,
 \Q \|u\|_{\bf A}^2=\int_\Omega \nabla u\cdot {\bf A}\nabla u\,.
\end{equation}
For vectors ${\bf y}\in {\mbb R}^m$, we introduce also the spaces ${\bf L}^2(\Omega)=(L_2(\Omega))^m$, $ {\bf H}(\Omega, {\mr{div}})=\{ {\bf y}\in {\bf L}^2(\Omega): {\mr{div}}\, {\bf y} \in L_2(\Omega)\}$,   ${\mbf W}_{h,p}(\Omega,{\mr{div}}) =\{{\bf y}\in {\bf H}(\Omega, {\mr{div}}): y_k\big|_{\tau_r}\in {\mc P}_p,\,\, \forall r\in {\mc R},\,\,k=1,2,\dots,m  \}$ and the norm
$]\!| {\bf y} |\![_{{\bf A}^{-1}}=(\int_\Omega {\bf A}^{-1} {\bf y}\cdot {\bf y})^{1/2}$.
\begin{theorem}\label{Th:Aubin-1}
 Let $f\in L_2(\Omega)$, $0< \sigma={\mr{const}} $, $\psi_D\in H^1(\Omega),\,\,\psi_N\in L_2(\Gamma_N)$, $v$ be any function of
$H^1(\Omega)$ that satisfies the boundary condition on $\Gamma_D$. Then, for any
${\bf z} \in {\bf H}(\Omega, {\mr{div}})$, satisfying on ${\Gamma_N}$
the boundary condition $ {\bf z}\cdot {\bi \nu}
=\psi_N$, we have
\begin{equation}\label{Aubin-1}
 |\!\!\,\!|\!\!\,\!| v-u|\!\,\!\!|\!\,\!\!|^2\le ]\!| {\bf A}\nabla\,v +{\bf z} |\![_{{\bf A}^{-1}}^2+
\frac{1}{\sigma}\|f-\sigma v - {\mr{div}} \,{\bf z}\|_{L_2(\Omega)}^2\,.
\end{equation}
\end{theorem}
  \begin{proof}
Estimate (\ref{Aubin-1}) is a special case of the results of Aubin \cite{Aubin:72}, see, e.g., Theorem 22 in Introduction
and Theorems 1.2, 1.4, 1.6 in Chapter 10.
\end{proof}
\par Obviously, if $ \sigma\rightarrow 0$  the majorant of Aubin loses precision and with $ \sigma= 0$ makes no sense.
 \par   If $ \sigma\equiv 0$, one can use the majorant of Repin and Frolov \cite{RepinFrolov:02}. Let
 for simplicity, $\Gamma_D=\bo \Omega$, $\psi_D \equiv 0$, ${\bf A}={\bf I}$, where ${\bf I}$ -- the identity matrix, and $ \sigma\equiv 0$. Then
\begin{equation}\label{rep-frol}
\|\nabla (v - u) \|_{{\bf L}^2(\Omega)}^2 \le (1+\epsilon)\|\nabla v+ {\bf
z} \|_{{\bf L}^2(\Omega)}^2 + c_\Omega(1+\frac{1}{\epsilon})\|\nabla\cdot {\bf z}- {
f}\|_{L_2(\Omega)}^2\,,\q \forall \,\,\epsilon>0\,,
\end{equation}
 where $v$ and ${\bf z}$ are a function and an arbitrary vector-function from $\7H^1(\Omega)$ and ${\bf H}(\Omega, {\mr{div}})$ respectively, and $c_\Omega$ is the constant from the Friedrichs inequality.

          \par It was shown in \cite{AnufrievKorneevKostylev:2006, Korneev:2011} that the correction of arbitrary vector-function ${\bf z} \in {\bf H}(\Omega, {\mr{div}})$ into the vector-function $\bi \tau$, satisfying the balance/equilibrium equations, can be done by quite a few
   rather simple techniques. In particular, it is true for the correction of the flux vector-function $\nabla u_{\mr{fem}}$ into $\bi \tau(u_{\mr{fem}})$.  This allows to implement the a posteriori bound  $|\!\!\,\!|\!\!\,\!| v-u|\!\,\!\!|\!\,\!\!|\le \| {\bf A}\nabla\,u_{\mr{fem}} +\bi \tau(u_{\mr{fem}}) \|_{{\bf A}^{-1}}^2$, see, {\em e.g.,} Mikhlin \cite{Mikhlin:64}., or the bound with the additional free vector-function
     in the right part, which we present below. For simplicity, we restrict considerations to the same homogeneous Dirichlet problem for the Poisson equation in a two-dimensional convex domain.  Let $T_k$ be the projection of the domain $\Omega$ on the axis $x_{3-k}$ and the equations of the left and lower parts of the boundary be $x_k=a_k(x_{3-k}),\,\, x_{3-k}\in T_k$.  If $\beta_k$
    are arbitrary bounded functions  and $\beta_1+\beta_2\equiv 1$,  then according to \cite{AnufrievKorneevKostylev:2006, Korneev:2011}
   \begin{equation}\label{yappi-1}
   \begin{array}{l}
    \|\nabla (v-  u) \|_{ {\bf L}_2(\Omega)} \le
      \|\nabla v + {\bf
    z} \|_{{\bf L}_2(\Omega)} +\\
    \vspace{-3mm}
    \\
      \sum_{k=1,2}\|\int_{a_k(x_{3-k})}^{x_k}
    \beta_k(f-\nabla \cdot{\bf z})
    (\eta_k,x_{3-k})\,d\eta_k \|_{ L_2(\Omega)}\,.
    \end{array}
   \end{equation}
\par In (\ref{yappi-1}) on the right  we have integrals from the residual and this hopefully will make the majorant more accurate.
  Besides there is an additional free function  $\beta_1 $  or $\beta_2 $  and it's  right choice (for instance, with the use of the found approximate solution $v$) can accelerate the process of the minimization of the right part, if such a process is implemented. If to estimate one-dimensional integrals under the sign of the $L_2$-norm, then we come to the bound similar to (\ref{rep-frol})

\par Some authors attempted to modify the majorant of (\ref{Aubin-1}) with the aim of achieving acceptable accuracy for all $ \sigma\ge 0$,
see {\em e.g.}, Repin and Sauter \cite{RepinSauter:2006} and Churilova \cite{Churilova:14}.
The majorant of the latter, defined for $\forall\,\sigma={\mr{const}}\ge 0$, has the form
\begin{equation}\label{Chur}
|\!\!\,\!|\!\!\,\!| v-u|\!\,\!\!|\!\,\!\!|^2\le (1+\epsilon)]\!| {\bf A}\nabla\,v +{\bf z} |\![_{{\bf A}^{-1}}^2+
 \frac{1}{\sigma+\frac{\epsilon}{c_\Omega(1+\epsilon)}}\|f-\sigma v {\mr{div}} \,{\bf z}\|_{L_2(\Omega)}^2\,.
\end{equation}

\par One of the efficient majorants for the finite element solutions was developed by Ainsworth and Vejchodsky \cite{AinsworthVejchodsky:2011, AinsworthVejchodsky:2015}. For its record, we  need additional  notations: $h_r$ is the diameter $\tau_r$, $\Pi_r^p: L_2(\tau_r)\rightarrow \mc P_p(\tau_r)$ is the operator of orthogonal projection  in $L_2(\tau_r)$,
and $\sigma_r={\mr{const}}$ is the value of $\sigma$ on $\tau_r$.
  The dependence of the constants on the data of the boundary value problem and finite element assemblage is much simpler,
  if the following condition is satisfied:
\par ${\mc A}$)  The domain $\Omega$ is a polygon in $ {\mbb R}^m, \,\,m= 2,3$, $\tau_r$ are compatible $m$-dimensional simplices (with flat faces and, respectively, straight edges)
forming triangulation of $\Omega$, satisfying  the conditions of  shape regularity.
  \par For simplicity  in  Theorems  \ref{Th:A-V} and \ref{Th:Ch-Fu-P-Vo}  below,    we additionally  assume $\Gamma_D=\bo \Omega$, $\psi_D \equiv 0$, ${\bf A}={\bf I}$.

   \begin{theorem}\label{Th:A-V}
 Let $u\in \7{H}^1(\Omega)$ be the weak solution of the problem
  and $u_{\mr{fem}} \in \7{\mbb V}(\Omega)$ be the solution by the finite element method.  Then there exists
    ${\bf z}\in {\mbf W}_{h,2}(\Omega,{\mr{div}}) $
       with the following properties:
  \par {\it i)} ${\bf z}$ is evaluated by the  patch wise  numerical procedure of linear numerical complexity,
  \par {\it ii)}    for all $x\in \tau_r$ and $r\in {\mc R}_*=\{r: \sqrt{\sigma}_rh_r< 1\}$  satisfies the equalities
     \begin{equation}\label{AiVe-0}
     \Pi_r^1  f -\sigma_r u_{\mr{fem}}+ {\mr{div}} {\bf z}=0\,,
        \end{equation}
 \par {\it iii)}  for the error $e_{\mr{fem}}=u-u_{\mr{fem}}$ and the error indicator $\eta_{\tau_r}({\bf z})$, defined as
 \begin{equation}\label{yappi-1111}
   \begin{array}{l}
  \eta_{\tau_r}^2({\bf z})=\|{\bf z}-\nabla u_{\mr{fem}}\|^2_{L_2(\tau_r)}, \q \forall r\in {\mc R}_*\,,
  \\
  \vspace{-3mm}
   \\
\eta_{\tau_r}^2({\bf z})=\|{\bf z}-\nabla  u_{\mr{fem}}\|^2_{L_2(\tau_r)} +
\frac{1}{\sigma_r} \|\Pi_K f-\sigma_r u_{\mr{fem}} + \mathrm{div}{\bf z}\|_{L_2(\tau_r)}^2\,,
    \q \forall   r\in {\mc R}\smallsetminus {\mc R}_*\,,
\end{array}
   \end{equation}
 there hold the bounds
\begin{equation}\label{AiVe-1}
 \vertiii{e_{\mr{fem}}}^2\le \sum_{\tau_r\in {\mc{T}}_h} \Big[
\eta_{\tau_r}({\bf z}) + \mathrm{osc}_{\tau_r}(f)\Big]^2\,,
\end{equation}
\begin{equation}\label{AiVe-2}
 \eta_\Omega^2({\bf z})=\sum_{r\in {\mc{R}}}
\eta_{\tau_r}^2({\bf z}) \le {\mbb C}\Big[ \vertiii{e_{\mr{fem}}}^2+ \sum_{r\in {\mc{R}}} \mathrm{osc}_{\tau_r}^2
(f)  \Big]\,,
\end{equation}
where\,    $\mathrm{osc}_{\tau_r}
(f) = \min\left\{\frac{h_r}{\pi},\frac{1}{\sqrt{\sigma_r}}\right\}\|f- \Pi_r^1f\|_{L_2(\tau_r)}$.
\end{theorem}
\begin{proof}  See Ainsworth and Vejchodsky \cite{AinsworthVejchodsky:2015} for the proof.    We note that in this work the bounds (\ref{AiVe-1}), (\ref{AiVe-2}) in a more general form are derived under more general conditions. In particular,  $\Gamma_N\neq \varnothing$,  the bound (\ref{AiVe-2}) is proved in the local version, {\em i.e.}, with $\eta_{\tau_r}^2({\bf z})$ on the left and with the restriction of the right part to the patch $\delta^{(r)}$.
\end{proof}
  \par Let us present one more majorant obtained by  Cheddadi {\em et al.} \cite{CheddadiRadekFucikPrietoVohralik:2009} for approximate solutions of the singularly perturbed reaction-diffusion problem by the method of vertex-centered finite volumes.  Let us introduce the notations: ${\mc D}_h$  is the dual in respect to  $ \mathcal{T}_h$ partition of $\Omega$;  ${\mc S}_h$ is the finer mesh, induced by the partition   ${\mc D}_h$; ${ D}$ is the polygon  with the center in the vertex of triangulation $ \mathcal{T}_h$ and  containing all simplices of the finer mesh  with this vertex, $h_{ D}$ is its diameter; ${\mc D}_h^{\mr{int}}$ is the set of all polygons  ${ D}$, for which $\bo{ D} \cap \bo \Omega=\varnothing$. For additional information  about these objects we refer to  \cite{CheddadiRadekFucikPrietoVohralik:2009}.

\begin{theorem}\label{Th:Ch-Fu-P-Vo}
    Let $u_h$ be the solution by the method of vertex-centered finite volumes, $e_h=u-u_h$,  vector-function ${\bf z} \in {\bf H}(\Omega, {\mr{div}})$ satisfy the equalities
 \begin{equation}\label{Chedd-1}
     (f- \nabla{\bf z}-\sigma u_h,1)_{ D}=0\,, \q \forall { D}\in {\mc D}_h^{\mr{int}}\,,
 \end{equation}
   and $\theta_D={\mr{min}}({C}_{ D}h_{ D}^2, \sigma_{ D}^{-1})$, where ${C}_{ D}$ is the constant from the Poincar\'{e} inequality for the polygon
    ${\mc D}$.  Then
 \begin{equation}\label{Chedd-2}
 \vertiii{e_h}^2\le \eta_\Omega^2({\bf z})=
    \sum_{{ D}\in {\mc D}_h}
 \Big[  \|\nabla u_h+{\bf z} \|_{L_2({D})}
 +\sqrt{\theta_D}\|f-\sigma u_h -\nabla\cdot {\bf z}\|_{L_2({ D})}
 \Big]^2\,.
\end{equation}
\end{theorem}
 \begin{proof}   Theorem is one of the results of \cite{CheddadiRadekFucikPrietoVohralik:2009}, see Theorem~4.5.
\end{proof}

\par Majorants in  (\ref{rep-frol}) --  (\ref{Chur}) have definite merits, but are not consistent at the application, for instance, to
  solutions by the finite element and other mesh methods.   If $v=u_{\mr{fem}}$  is the finite element solution
  to the problem  (\ref{poiss}) at
    $\Gamma_D=\bo \Omega$, $\psi_D \equiv 0$, ${\bf A}={\bf I}$,  at $ \sigma= 0$, then we can use (\ref{rep-frol}). For our purpose it is sufficient to consider the approximate solutions from the space ${\mbb V}_{h}(\Omega) ={\mbb V}_{h,p}^1(\Omega) \!\!\subset\!\! C^1(\Omega)\!\cap\!  H^2(\Omega)$. Under the assumption $u\in H^l(\Omega)$, we have  a priori error bounds
   \begin{equation}\label{not improve}
   \|u-v\|_{H^k(\Omega)}\le c h^{l-k}\|u\|_{H^l(\Omega)}\,, \q k=1,2,\q k\le l\le p+1\,.
  \end{equation}
 In particular, if $f\in L_2(\Omega)$ and consequently $l=2$, see below (\ref{2-f}), then according to (\ref{not improve}) the left part of (\ref{rep-frol})
  is estimated as ${\mc O}(h^2)$. At the same time at the choice ${\bf z}=-\nabla v$ the first term of the right part of (\ref{rep-frol})
  vanishes, but $ c_\Omega(1+\frac{1}{\epsilon})\|\nabla\cdot {\bf z}- {
    f}\|_{L_2(\Omega)}^2$ at any $\epsilon>0$ can be dounded only by a constant.  The bounds (\ref{not improve})  are {\em not impovable} in the order for $u\in H^2(\Omega)$.  For  $\Omega \in {\mbb R}^2$ this was proved  by Oganesian and Ruhovets \cite{OganesianRuhovets:1979} by estimating the corresponding Kolmogorov's width.  From their resullts and the results on the regularity of the solutions of (\ref{poiss}),  it follows existence of such $f\in L_2(\Omega)$ that $u\in H^2(\Omega)$
    and the second summand on the right of (\ref{rep-frol}) is estimated by the constant \textit{from below}.  In other words the orders of smallness of the left and the right parts of the a posteriori bound are different, and the value ${\mc O}(h^2)$ on the left is estimated
    by the right part only with the order of unity. If $l>2$, the  left and the right parts  are estimated with not equal orders ${\mc O}(h^{2(l-k)})$ and ${\mc O}(h^{2(l-k-1)})$.
   \par Inconsistency of the majorant (\ref{Chur})  at $\sigma\le c h^{-\alpha},\,\, 0\le \alpha <2,\,\, c={\mr{const}}$, is established in
   a similar way.  Majorant (\ref{yappi-1}) is in general  also inconsistent. The inconsistency of the above mentioned majorants, obviously, is retained, if finite elements of the class $C$  are used and the test flax is found by some recovery procedure, satisfying only the requirements  ${\bi \alpha}$), ${\bi \gamma}$).
    \par  The equality of the orders of smallness of the left and right parts of the  majorants (\ref{AiVe-1}) and  (\ref{Chedd-2}) is well provided, as follows, {\em e.g.},  from (\ref{AiVe-2}) and similar bound, proved in \cite{CheddadiRadekFucikPrietoVohralik:2009}.
     However, it is achieved only for the test fluxes, satisfying the additional conditions reflecting the requirement ${\bi \beta}$), see (\ref{AiVe-0}) and (\ref{Chedd-1}).  For this reason these majorants
     might be called conditionally consistent.

 \section{\normalsize Modified Aubin's  a posteriori error majorant rodust for all $\sigma\ge 0$
 } \label{Se:consist}  

 \par
    In this Section we derive the guaranteed, reliable  and robust majorant\footnote{Definitions of the terms {guaranteed, reliable, locally effective  etc., used in relation to the a posteriori error estimates,   can be found in  \cite{AinsworthVejchodsky:2015}}.}, which is well defined for all $\sigma \ge 0$.
    More over, it will be shown that it is consistent for the finite element solutions of the problems with  sufficiently smooth data, see Section~\ref{Se:cons-inverse}.  The new majorant will coincide with the Aubin's majorant for $\sigma \ge \sigma_*$, where  $ \sigma_*$ is some critical value, which can be differently defined for different numerical methods and different ways of the derivation of a posteriori bounds. In general,  when $v \in \7H^1_{\Gamma_D}(\Omega)$  is any approximation for $u$,  $\sigma_*$   can be defined from the inequality
  \begin{equation}\label{Aub-1}
   \frac{\| u - v  \|_{{\bf A}} ^2}{\hspace{5.5mm} \| u - v \|_{L_2(\Omega)} ^2}   \ge  \sigma_*  > 0\,.
   \end{equation}
   \par In some situations  this ineqality can be relaxed. Suppose that $v=u_G$   is the approximate solution by the Galerkin  method in the subspace ${\mc V}(\Omega)\subset  \7H^1_{\Gamma_D}(\Omega,\Delta)$. Then it can be adopted that $\sigma_*$ satisfies
  \begin{equation}\label{Aub-100}
  \frac{\| u - v \|_{{\bf A}} ^2}{\hspace{5.5mm}\| u - Qu \|_{L_2(\Omega)}^2}   \ge  \sigma_*  > 0\,.
   \end{equation}
 where $Q$ is the operator of orthogonal projection $ L_2(\Omega)\rightarrow{\mc V}(\Omega)$,  {\em i.e.}, such that for  $\forall \phi\in L_2(\Omega)$ we have
  $$
       (Q\phi, \psi)_\Omega=(\phi, \psi)_\Omega\,,  \quad \forall\, \psi \in {\mc V}(\Omega)\,.
  $$
\par  Inequality (\ref{Aub-100}) can be also  used for $v$ from any subspace  ${\mc V}(\Omega)\subset  \7H^1_{\Gamma_D}(\Omega)$,
  but in this case additional  conditions on the test flax ${\bf z}$, arise. It is sufficient,
for instance, that   ${\bf z}$ satisfied to the equalities (\ref{G-bound-As}).

      \par Let  $\hat{f}(x)=\Pi_r^1 f$ for $ x\in \tau_r, \,\, r=1,2,\dots,{\mc R}$. In Theorem below  domains $\tau_r$  can be understood as arbitrary convex subdomains of some partition of the domain
  $$
      {\Omega}={\mr{interior}} \{\,\bigcup_{1}^{\mc R} \OL{\tau}_r \}\,, \quad \tau_r\cup \tau_{r'}=\varnothing,\,\,r\neq r',
      \quad {\mr{diam}}[\tau_r]=h_r\,,
 \vspace{-2mm}
  $$
  for which  the Poincar\'{e} inequalities hold, see, {\em e.g.},    Nazarov and Poborchi \cite{NazarovPoborchi:2012},
    $$
     \inf _{c\in {\mbb R}} \,\|\phi-c\|_{L_2(\tau_r)}\le \frac{h_r}{\pi} |\phi|_{H^1(\tau_r)}\,, \quad \phi \in H^1(\tau_r)\,.
    $$
 \begin{theorem}\label{Th: K3}
   Let $\Gamma_D=\bo \Omega$, conditions of Theorem~\ref{Th:Aubin-1} be fulfilled, and
    $\sigma_*$  satisfies the inequality (\ref{Aub-1}). Then
 \begin{equation}\label{K-23}
               |\!\!\,\!|\!\!\,\!| v-u|\!\,\!\!|\!\,\!\!|^2\le \Theta  {\mc M} (\sigma,\sigma_*,f,v,{\bf z})\,,
   \end{equation}
  where
      \begin{equation}\label{K-23+}
                     {\mc M} (\sigma,\sigma_*,f,v,{\bf z})=       \Big[ ]\!| {\bf A}\nabla\,v +{\bf z} |\![_{{\bf A}^{-1}}^2+
            \theta    \|f-\sigma v - {\mr{div}} \,{\bf z}\|_{L_2(\Omega)}^2 \Big]\,,
 \end{equation}
   and
 \begin{equation}\label{K-23-1}
 \begin{array}{ll}
  \Theta=\left\{  \begin{array}{ll}
   2/(1+\kappa)\, , &\forall\, \sigma\in [0,\sigma_*] \\
    \vspace{-5mm}
    \\
    1 \,,   &\forall\, \sigma> \sigma_*
    \end{array} \right\}, &
     \theta=\left\{  \begin{array}{ll}
 1/\sigma_* \, ,    &\forall\, \sigma\in [0,\sigma_*] \\
    \vspace{-5mm}
    \\
 1/\sigma \, ,    &\forall\, \sigma> \sigma_*
 \end{array} \right\}\,.
 \end{array}
 \end{equation}
   with $\kappa=\sigma / \sigma_*$. Besides, for $\sigma \in [0,\sigma_*]$ and $\sigma \ge \sigma_*$,  respectively, we have the bounds
 \begin{equation}\label{K-23-2}
 \begin{array}{c}
  |\!\!\,\!|\!\!\,\!| v-u|\!\,\!\!|\!\,\!\!|^2\le  \Theta_1  {\mc M}(\sigma,\sigma_*,\hat{f},v,{\bf z})
   +      \sum_r \frac{h_r^2}{\varepsilon\pi^2}\int_{\tau_r} (f-\Pi_r^1 f)^2dx\,,  \q \forall\,\varepsilon >0\,,
     \end{array}
\end{equation}
 \begin{equation}\label{K-23-3}
 \begin{array}{c}
  |\!\!\,\!|\!\!\,\!| v-u|\!\,\!\!|\!\,\!\!|^2\le
    \Theta_2 {\mc M} (\sigma,\sigma_*,\hat{f},v,{\bf z})   +  \sum_r \frac{1}{\sigma}\int_{\tau_r} (f-\Pi_r^1 f)^2dx\,,
     \end{array}
\end{equation}
 where
 \begin{equation*}\label{K-23-4}
  \Theta_1=\left\{  \begin{array}{ll}
   (2+\varepsilon)/(1+\kappa)\,, \q &   0\le\sigma \le \sigma_*/(1+\varepsilon)\,,\\
    \vspace{-5mm}
    \\
    1 +\varepsilon \,,  &  \sigma_*/(1+\varepsilon) \le \sigma\le \sigma_*\,,
    \end{array} \right.
 \end{equation*}
 and
 $$
 \Theta_2=1+ \frac{1}{1+\kappa^{-1}}\,.
 $$
    If $v=u_G$ is the approximate solution  by the method of Galerkin in the space ${\mc V}(\Omega)\subset  \7H^1_{\Gamma_D}(\Omega,\Delta)$,  then the bound (\ref{K-23})-(\ref{K-23-1})  takes place with $\sigma_*$, satisfying to the inequality (\ref{Aub-100}) and $ {\bf z}={\bf z}_G:=- {\bf A} \nabla \,u_G$, {\em i.e.},
 \begin{equation}\label{Koo}
 \begin{array}{c}
               |\!\!\,\!|\!\!\,\!| u_G-u|\!\,\!\!|\!\,\!\!|^2\le \Theta  {\mc M}_G (\sigma,\sigma_*,f,u_G,{\bf z}_G)\,,
               \\
    \vspace{-3mm}
       \\
                \quad  {\mc M}_G (\sigma,\sigma_*,f,u_G,{\bf z}_G)=\theta    \|f-\sigma u_G - {\mr{div}} \,{\bf z}_G\|_{L_2(\Omega)}^2 \,.
    \end{array}
   \end{equation}
 \end{theorem}
  \begin{proof}  Obviously, for $\sigma \ge \sigma_*$ the majorant (\ref{K-23})-(\ref{K-23-1}) concides with the majorant (\ref{Aubin-1}),
  whereas  for $\sigma \in [0,\sigma_*]$ majorants (\ref{K-23})-(\ref{K-23-1}) and (\ref{Aubin-1}) are signuficantly different.  Therefore,
  it is necessary to consider only the case $\sigma \le \sigma_*$. In order not to encumber the proof with secondary details, we  assume in the proof that ${\bf A}={\bf I}$ and $\psi_D \equiv 0$.
  \par  For the solution $u$ of the problem and arbitrary $v\in \7{H}^1(\Omega)$ and ${\bf z} \in {\bf H}(\Omega, {\mr{div}})$,  we have
    \begin{equation}\label{bound-1}
\begin{array}{c}
    |\!\!\,\!|\!\!\,\!| v-u|\!\,\!\!|\!\,\!\!|^2= \int_\Omega \big[{ \nabla} (v-u)\cdot \nabla (v-u) +
    \sigma (v-u)(v-u) \big]= \\  \vspace{-3mm}
    \\= \int_\Omega \big[(\nabla v+{\bf z})\cdot \nabla (v-u) -  ({\bf z}+\nabla u)\cdot
    \nabla (v-u)+  
     \sigma (v-u)(v-u) \big]\,.
  \end{array}
\end{equation}
  Integrating by parts the second term on the right and using the inequality
   \begin{equation}\label{simple}
   a_1b_1+a_2b_2\le (a_1^2+\frac{1}{\sigma_*}a_2^2)^{1/2}(b_1^2+{\sigma_*}b_2^2)^{1/2}\,,
  \end{equation}
   we find that
 \begin{equation}\label{K-25}
 \begin{array}{c}
    |\!\!\,\!|\!\!\,\!| v-u|\!\,\!\!|\!\,\!\!|^2  =\| \nabla (u-v)\|_{{\bf L}^2(\Omega)}^2+\sigma\|u-v\|_{ L_2(\Omega)}^2\le \Big[ \| \nabla v-{\bf z}\|_{{\bf L}^2(\Omega)}^2  +
   \\
    \vspace*{-3mm}
   \\
     +\frac{1}{\sigma_*}\|f-\sigma v - {\mr{div}} \,{\bf z}\|_{L_2(\Omega)}^2\Big]^{1/2}\times
    \Big[ \| \nabla (u-v)\|_{{\bf L}^2(\Omega)}^2+\sigma_*\|u-v\|_{ L_2(\Omega)}^2 \Big]^{1/2}\,.
  \end{array}
   \end{equation}
  At $\beta \in (0,1]$ the inequality (\ref{Aub-1})   allow us to transform the second multiplier in the right part of (\ref{K-25}) to the form  \begin{equation}\label{K-26}
  \begin{array}{c}
   \| \nabla (u-v)\|_{{\bf L}^2(\Omega)}^2+\sigma_*\|u-v\|_{ L_2(\Omega)}^2=|\!\!\,\!|\!\!\,\!| u-v|\!\,\!\!|\!\,\!\!|^2+
   (\sigma_*-\sigma)\|u-v\|_{ L_2(\Omega)}^2\le \\
    \vspace*{-3mm}
   \\
 |\!\!\,\!|\!\!\,\!| u-v|\!\,\!\!|\!\,\!\!|^2+
   (\sigma_*-\sigma)\big[ \frac{\beta}{ \sigma_*}\| \nabla (u-v)\|_{{\bf L}^2(\Omega)}^2 + (1-\beta)\|u-v\|_{ L_2(\Omega)}^2 \big] =\\
    \vspace*{-3mm}
   \\
 \big[ 1+(\sigma_*-\sigma)\frac{\beta}{\sigma_*}  \big]\| \nabla (u-v)\|_{{\bf L}^2(\Omega)}^2+\big[ (1-\beta)(\sigma_*-\sigma)+\sigma \big]\|u-v\|_{ L_2(\Omega)}^2\,.
   \end{array}
 \end{equation}
    The choice
    $
    \beta= {1}/(1+\kappa)
    $
    makes the relation of the multipliers before  second and  first norms on the right of (\ref{K-26}) equal to $\sigma$.
      Substituting
     such  $ \beta$  in (\ref{K-26}) and then  (\ref{K-26}) in  (\ref{K-25}) leads to the inequality
   \begin{equation}\label{K-27}
       |\!\!\,\!|\!\!\,\!| v-u|\!\,\!\!|\!\,\!\!|^2 \le \frac {2}{1+\kappa} \Big[ \| \nabla v-{\bf z}\|_{{\bf L}^2(\Omega)}^2     +\frac{1}{\sigma_*}\|f-\sigma v - {\mr{div}} \,{\bf z}\|_{L_2(\Omega)}^2\Big]^{1/2} |\!\!\,\!|\!\!\,\!| v-u|\!\,\!\!|\!\,\!\!| \, ,
         \end{equation}
   which is equivalent to  (\ref{K-23}) at ${\bf A}={\bf I}$.
    \par In order to prove (\ref{K-23-2}), we transform (\ref{bound-1}) to the form
    \begin{equation}\label{pre-bound}
\begin{array}{c}
    |\!\!\,\!|\!\!\,\!| v-u|\!\,\!\!|\!\,\!\!|^2= \int_\Omega \big[{\bi \nabla} (v-u)\cdot \nabla (v-u) +
    \sigma (v-u)(v-u) \big]dx= \\  \vspace{-3mm}
    \\=  \int_\Omega \big\{(\nabla v+{\bf z})\cdot \nabla (v-u) +  [\nabla\cdot{\bf z}+\Delta u +\sigma (v-u)  ]
    \ (v-u)
     \big\} dx = \\  \vspace{-3mm}
    \\
    = \int_\Omega \big\{(\nabla v+{\bf z})\cdot \nabla (v-u) +  [\nabla\cdot{\bf z} - \hat{f}+\sigma v ]
     (v-u) + (\hat{f}-f) (v-u)
     \big\} dx
        \,.
   \end{array}
\end{equation}
   We represent the integral of the last summand in the right part of (\ref{pre-bound})  by the sum of the integrals over
  subdomains $\tau_r,\,\, r=1,2,\dots,{\mc R}$, and each estimate with the help of Poincar\'{e} inequality:
    \begin{equation}\label{Poin}
     \begin{array}{c}
   \int_{\tau_r} (\hat{f}-f) (v-u)dx=\int_{\tau_r} (\hat{f}-f) (v-u-{\mr{c}})dx\le  \\  \vspace{-3mm}
    \\
   \le  \|\hat{f}-f\|_{L_2(\tau_r)} \inf_{{\mr{c}}\in {\mbb R}} \|v-u-{\mr{c}}\|_{L_2(\tau_r)} \le  \\  \vspace{-3mm}
    \\  \le  \frac{d_r}{\varepsilon\pi} \|\hat{f}-f\|_{L_2(\tau_r)}   (\varepsilon|v-u|_{H^1(\tau_r)} )  \,.
   \end{array}
   \end{equation}
  Now for getting (\ref{K-23-2}) it is sufficient to substitute (\ref{Poin}) in (\ref{pre-bound}), to apply to the right part the Cauchy inequality
  and than to use the inequality (\ref{K-26})  with $ \beta$ satisfying the condition
    $
    \sigma[1+ \varepsilon +(\sigma_*-\sigma)\frac{\beta}{\sigma_*} ] =  (1-\beta)(\sigma_*-\sigma)+\sigma
    $.
 \par  The inequality (\ref{K-23-3}) follows for
  $$
  \beta=\frac{1}{2}\big(1+\frac{1}{1+ \kappa^{-1}} \big)\,,
  $$
 from the estimates
 \begin{equation}\label{25}
 \begin{array}{l}
  \hspace*{1mm}   |\!\!\,\!|\!\!\,\!| v-u|\!\,\!\!|\!\,\!\!|^2  \le \Big[ \| \nabla v-{\bf z}\|_{{\bf L}^2(\Omega)}^2  +
       +\frac{1}{\sigma}\|\hat{f}-\sigma v - {\mr{div}} \,{\bf z}\|_{L_2(\Omega)}^2+ \frac{1}{\sigma}  \|\hat{f}-f\|_{L_2(\Omega)}^2 \Big]^{1/2}\times
     \\
    \vspace*{-3mm}
   \\
          \hspace*{14.2mm} \times \,\Big[ \| \nabla (u-v)\|_{{\bf L}^2(\Omega)}^2+2\sigma\|u-v\|_{ L_2(\Omega)}^2 \Big]^{1/2}\,,
     \\  \vspace*{-3mm}
       \\
    \|u-v\|_{ L_2(\Omega)}^2\le \beta  \|u-v\|_{ L_2(\Omega)}^2+(1-\beta)\frac{1}{\sigma_*}\|\nabla(u-v)\|_{ {\bf L}^2(\Omega)}^2\,,
   \end{array}
   \end{equation}
  second of which is obtained with the use of   (\ref{Aub-1}).

  \par  Let us turn to the proof of the bound (\ref{Koo}).  For $e=v-u$ and ${\bf z} \in {\bf H}(\Omega, {\mr{div}})$  we have
    \begin{equation}\label{G-bound}
    |\!\!\,\!|\!\!\,\!| \,e\,|\!\,\!\!|\!\,\!\!|^2=
     \int_\Omega \big\{(\nabla v+{\bf z})\cdot \nabla e +  [  {f}-\sigma v - \nabla\cdot{\bf z} ]\,  e   \big\} dx
        \,.
\end{equation}
    Suppose,
       ${\bf z}$    satisfies the identity
        \begin{equation}\label{G-bound-As}
        \int_\Omega  (  {f}-\sigma v - \nabla\cdot{\bf z} )\,
        \psi
      dx\,, \quad \forall \, \psi\in {\mc V}(\Omega)\,,
       \end{equation}
    on the finite element space  ${\mc V}(\Omega)$.   If  $v\in{\mc V}(\Omega)$, then for $e_\circ=(Qu-u)$  the equalities
    $$
         e-Qe=v-u-Qv+ Qu =  Qu-u=e_\circ
         $$
    hold and
    \begin{equation}\label{G-bound+1}
\begin{array}{c}
    \int_\Omega   [  {f}-\sigma v - \nabla\cdot{\bf z} ] e
     dx   = \beta \int_\Omega   [  {f}-\sigma v - \nabla\cdot{\bf z} ]
     (e-Qe)
     dx +\\
     \vspace*{-2mm}
       \\
     (1-  \beta) \int_\Omega   [  {f}-\sigma v - \nabla\cdot{\bf z} ] e  dx =
         \beta \int_\Omega   [  {f}-\sigma v - \nabla\cdot{\bf z} ]\,
       e_\circ  dx+
       \\ \vspace*{-2mm}
       \\
       (1-  \beta) \int_\Omega   [  {f}-\sigma v - \nabla\cdot{\bf z} ] e   dx
        \,.
  \end{array}
\end{equation}
   Substituting  (\ref{G-bound+1}) in (\ref{G-bound}), and applying the Cauchy inequality and the  inequality
     (\ref{Aub-100}), we get
   \begin{equation}\label{G-bound+2}
    \begin{array}{c}
    |\!\!\,\!|\!\!\,\!| \,e\,|\!\,\!\!|\!\,\!\!|^2\le
    =  \big\{ \|\nabla v+{\bf z})\|^2 +  \big[\frac{\beta}{\sigma_*}
    +
     \frac{1-\beta}{\sigma}\big]\| {f}-\sigma v - \nabla\cdot{\bf z} \|^2
     \big\} ^{1/2} \times
      \\ \vspace*{-3mm}
       \\
     \big\{ (1+\beta)\|e\|_0^2    +(1-\beta)  \sigma \|e\|_0^2  \big\} ^{1/2}
        \,.
    \end{array}
\end{equation}
   \par The use of  $\beta=(1-\kappa)/(1+\kappa)$ leads to the estimate coinciding formally with  (\ref{K-23})-(\ref{K-23-1}),
   but with $\sigma_*$, satisfying the inequality  (\ref{Aub-100}).  Now we note that when ${\mc V}(\Omega) \subset  \7H^1_{\Gamma_D}(\Omega,\Delta)$ the identity  (\ref{G-bound-As}) holds.  Besides, for ${\bf z}= -{\bf A}\nabla u_G$, now well defined,
   the norm $]\!| {\bf A}\nabla\,v +{\bf z} |\![_{{\bf A}^{-1}}$ equals zero, that proves the bound (\ref{Koo}).
\end{proof}
   \begin{remark}\label{Re-1}
   There are other ways of obtaining  majorants similar to (\ref{K-23})--(\ref{K-23-1}) and  (\ref{K-23-2})--(\ref{K-23-3}).
     For simlicity, let ${\bf A}\equiv {\bf I}$ and $\Gamma_D=\bo \Omega$.  We can consider the subsidiary problem
 \begin{equation}\label{f-lambda}
  - \Delta u +\Bbbk\, u=f_{\lambda,\sigma}\,,\q x\in\Omega\,, \q u\big|_{\bo \Omega}=0\,,
\end{equation}
   with an arbitrary $\Bbbk \ge \sigma_*$ and $f_{\lambda,\sigma}=f+(\Bbbk-\sigma)u$, whose solution is the same with the problem
    (\ref{poiss}) at  the specified  ${\bf A}$ and $\Gamma_D$.  For the approximation $v$ of the solution to the problem (\ref{f-lambda}) in Aubin's majorant related to (\ref{f-lambda}) one can use
     the approximation of the problem (\ref{poiss}).  After substitution of $f_{\lambda,\sigma}=f+(\Bbbk-\sigma)u$ in the Aubin's majorant,
   application of the Cauchy inequality with $\varepsilon$, and some manipulations,  we come to the subsidiary majorant.  By minimization of it in the respect of $\Bbbk$, $\beta$ and $\varepsilon$, we come to the set majorants including similar to those in Theorem~\ref{Th: K3}. Obviously, by changing  the choice  of $\lambda,\,\,\beta$, and $\varepsilon$, we can  change the weights before the first and second norms in the right parts of the majorants.
       \end{remark}

  \section{\normalsize Consistent a posteriori  majorants for finite element method errors
      } \label{Se:FEM}  

   For specific classes of approximate solutions and, in particular, for solutions  by the finite element method, the critical values $\sigma_*$ of the reaction coefficient in the derived error majorants  can be estimated.
   \begin{lemma}  \label{Le: Kor}
   Let $\Gamma_D=\bo \Omega$, $\sigma\equiv {\mr{const}}$, $\psi_D \equiv 0$,   $u\in \7{H}^1(\Omega)$, $f\in H^{-1}(\Omega)$, the finite element assemblage generates the space
   ${\mbb V}_{h,p}(\Omega),\,\,p\ge 1$,   and  $e_{\mr{fem}}=u_{\mr{fem}}-u$.
    Then
 \begin{equation}\label{A-Nitsche}
     \|e_{\mr{fem}} \|_{0}\le c_\dag h \|e_{\mr{fem}} \|_{\bf A}\,, \quad c_\dag=\frac{\sqrt{\mu_2}}{\mu_1}  c_{\circ}c_{\mr{ap}}\,,
\end{equation}
  with the constants $c_{\circ},\,\,c_{\mr{ap}}$,  defined below,   see. (\ref{2-f}), (\ref{Appr-phi}).
\end{lemma}
   \begin{proof}   Let us consider the problem of finding the solution  $\chi\in \7{H}^1(\Omega)$ of the integral identity
   \begin{equation}\label{IntT}
       a_{\Omega}(\chi,v)+\sigma(\chi,v)_\Omega=<F,v>\,, \quad \forall\, v\in \7{H}^1(\Omega)\,,
 \end{equation}
 where
$$
       a_{\Omega}(v,w)=\int_{\Omega} \nabla v\cdot {\bf A}\nabla w \,dx\,.
       $$
  If $\Omega$ is sufficiently smooth and  $\sigma \ge 0$, then
  \begin{equation}\label{2-f}
   \| \chi \|_{2} \le c_\circ \|F \|_{0}\,,\Q
   c_\circ=c_\circ(\Omega)={\mr{const}}\,,
  \end{equation}
     with any  $F\in L_2(\Omega)$. To prove this, we note that  for  $\sigma \le 1$ the inequality holds, see Ladyzhenskaya  and Uraltseva \cite{LadyzhenskayaUraltseva:73}.
    Obviously,  at $\sigma \ge 1$ we have
     \begin{equation}\label{u-sigma-f}
     \|\chi \|_{0}\le \sigma^{-1}\|F\|_{0}
     \,,
\end{equation}
      and  by  (\ref{2-f}) for the problem
           \begin{equation}\label{u-sigma-f_1}
          - {\mr{div }} ({\bf A} {\mr{grad}} \, \chi)  = F_\sigma\,, \q  F_\sigma=F-\sigma\,\chi \,, \q \chi\big|_{\bo
  \Omega}=0\,,
          \end{equation}
           it follows that
         \begin{equation} \label{u-sigma-f+}
          \| \chi \|_{2}\le c_\circ \|F_\sigma\|_{0}\le c_\circ (\|F \|_{0}+
            \sigma \|\chi \|_{0})\le 2 c_\circ \|F
            \|_{0}\,.
            \end{equation}
          It is left only to redefine the constant in (\ref{2-f}).
             \par
      Lt us introduce the notations $u_\circ,\,  u_{\mr{fe}}$
  and  $  u_s$  for the functions minimizing
  $\| u-\phi\|_{0}^2,\,\| u-\phi\|_{\bf A}^2$, and
  %
  $
  h^{-2}\| u-\phi\|_{0,\Omega}^2+\| u-\phi\|_{\bf A}^2\,,
  $
   recpectively, among all  $\phi\in \7{\mbb V}_h(\Omega)$ and the notations for the respective errors
     $e_\circ=u_\circ -u,\,\,u_{\mr{fe}}=u_{\mr{fe}}-u $  and $e_s=u_s -u$.
 Since $u_{\mr{fem}}$ minimizes  $|\!\!\,\!|\!\!\,\!| u -\phi|\!\,\!\!|\!\,\!\!|^2,\,\, \phi\in \7{\mbb
  V}_h(\Omega)$, we conclude that
  \begin{equation}\label{L_2_е}
     \|e_{\mr{fem}}\|_{\bf A}^2+\sigma \|e_{\mr{fem}}\|_0^2 \le \|u-\wt{u} \|_{\bf A}^2+\sigma \|u-\wt{u}\|_{0}^2\,,
    \end{equation}
  where $\wt{u}$ can be any from functions $\wt{u}=u_\circ,\,  u_{\mr{fe}},\,u_s$. If to take into attention the inequalities
    $\|e_{\mr{fem}}\|_{0}\ge \|e_{\circ}\|_{0}$ and $\|e_{\mr{fem}}\|_{\bf A}\ge \|e_{\mr{fe}}\|_{\bf A}$, following from the definitions of fun   $u_{\circ}$ and $u_{\mr{fe}}$,  then
 (\ref{L_2_е}) implies
 \begin{equation}\label{L_2_K}
 \begin{array}{c}
  \|e_{\mr{fem}}\|_{0}\le \|e_{\mr{fe}} \|_{0}\,,\\
             \vspace*{-3mm}
             \\
  \|e_{\mr{fem}} \|_{\bf A}\le \|e_{\circ}\|_{\bf A}\,.
  \end{array}
 \end{equation}

   \par Let    $\phi \in \7{H}^1(\Omega)$ be the solution of the problem
   \begin{equation}\label{poisson}
            a_{\Omega}(v,\phi)= (v,e_{\mr{fe}})_\Omega\,, \quad \forall\, v\in \7{H}^1(\Omega).
   \end{equation}
    Obviously,  $e_{\mr{fe}}\in L_2(\Omega)$ and as a consequence of  (\ref{2-f}) and symmetry of the bilinear form
    $a_{\Omega}(\cdot,\cdot)$, we have $\phi \in H^2(\Omega)$ and
    $$
               \|\phi \|_{2} \le c_{\circ}\|e_{\mr{fe}} \|_{0} \,.
    $$
     let us approximate $\phi $ by some function $\phi_{\mr{ap}}\in \7{\mbb V}_{h,p}(\Omega) $.
       We can use $\phi_{\mr{ap}}\in \7{\mbb V}_{h,1}(\Omega) \subset \7{\mbb V}_{h,p}(\Omega) $, and, in the case when the condition
       $\mc A)$ is fulfilled, obtain it by means of the quasi-interpolation operator of Scott and Zhang \cite{ScottZhang:90}.  In general,
       $\phi_{\mr{ap}}$  can be understood as the finite element solution, or  $L_2$-projection, or interpolation from $\7{\mbb V}_{h,1}(\Omega)$  or $\7{\mbb V}_{h,p}(\Omega) $. Let us underline that since in Lemma we only use the constant $c_{\mr{ap}}$ in the inequality
   \begin{equation}\label{Appr-phi}
     |\phi- \phi_{\mr{ap}}|_1^2\le c_{\mr{ap}}^2h^2\|\phi\|_2^2\le c_{\circ}^2c_{\mr{ap}}^2 h^2\|e_{\mr{fe}} \|_{0}^2\,,
    \end{equation}
       but not the very function $\phi_{\mr{ap}}$,   we can imply by $\phi_{\mr{ap}}$ any of the listed approximations which provides the better value of the constant.

   \par  Estimating $\|e_{\mr{fe}} \|_{0}$ by means of the Aubin-Nitsche trick   \cite{Aubin:72} for the problem (\ref{poisson}) and using the bound (\ref{Appr-phi}), we get:
 \begin{equation}\label{L_2-H^1}
 \begin{array}{c}
    \|e_{\mr{fe}} \|_{0}^2=\frac{1}{\mu_1}a_{\Omega}(e_{\mr{fe}}, \phi)\le\frac{1}{\mu_1}\inf_{w\in \7{\mbb V}_h(\Omega)} | a_{\Omega}(e_{\mr{fe}}, \phi-w)| \le\\
             \vspace*{-3mm}
             \\
     \frac{\sqrt{\mu_2}}{\mu_1} \|e_{\mr{fe}} \|_{\bf A}\inf_{w\in \7{\mbb V}_h(\Omega)} |\phi-w |_{1}\le
    \frac{\sqrt{\mu_2}}{\mu_1} \|e_{\mr{fe}} \|_{\bf A} |\phi - \phi_{\mr{ap}}|_{2}\le\\
             \vspace*{-3mm}
             \\
             \frac{\sqrt{\mu_2}}{\mu_1}  c_{\circ}c_{\mr{ap}} h\|e_{\mr{fe}} \|_{\bf A}\|e_{\mr{fe}} \|_{0}\,.
  \end{array}
  \end{equation}
   This  bound together with the inequality (\ref{L_2_K}) and the definitions  of functions $e_{\mr{fe}} ,\,e_{\mr{fem}} $
   results in the bound (\ref{A-Nitsche}).
      \end{proof}
        \begin{theorem}  \label{Th: Kor}
   Let $\Gamma_D=\bo \Omega$,  $\psi_D \equiv 0$,  and $u\in \7{H}^1(\Omega, \Delta)$. Let also the finite element assemblage generate the space $\7{\mbb V}_{h,p}^{\,0}(\Omega) \subset \7{H}^1(\Omega),\,\,p\ge 1,$ and $u_{\mr{fem}}$
   be the solution by the finite element method.  Then for $\sigma$ satisfying $0\le \sigma\le \sigma_*=1/(c_\dag h)^2$, where
   $c_\dag=\frac{\sqrt{\mu_2}}{\mu_1}  c_{\circ}c_{\mr{ap}}$,  and any
  ${\bf z} \in {\bf H}(\Omega, {\mr{div}})$
 \begin{equation}\label{K-fem+1}
    \begin{array}{c}
    |\!\!\,\!|\!\!\,\!| e_{\mr{fem}}|\!\,\!\!|\!\,\!\!|^2\le \frac{2}{1+{\mr{c}}_\dag^{2} h^2\sigma  }{\mc M}_{\mr{fem}}^{(1)} (\sigma,{f},{\bf z})
        \,,
                \\
    \vspace*{-3mm}
   \\
            {\mc M}_{\mr{fem}}^{(1)} (\sigma,f,{\bf z})=      \Big[ ]\!| {\bf A}\nabla\,u_{\mr{fem}}+{\bf z} |\![_{{\bf A}^{-1}}^2+
            {\mr{c}}_\dag^{2} h^2    \|{f}-\sigma u_{\mr{fem}} - {\mr{div}} \,{\bf z}\|_{L_2(\Omega)}^2 \Big]\,.
     \end{array}
     \end{equation}
    Under the condition
    ${\mc A}$), for   $\sigma \le \sigma_*/(1+\varepsilon)$ it holds also the bound
 \begin{equation}\label{K-fem+2}
               |\!\!\,\!|\!\!\,\!| e_{\mr{fem}}|\!\,\!\!|\!\,\!\!|^2\le \frac{2+\varepsilon}{1+{\mr{c}}_\dag^{2} h^2\sigma  }{\mc M}_{\mr{fem}}^{(1)} (\sigma,\hat{f},{\bf z})
         +      \sum_r \frac{h_r^2}{\varepsilon\pi^2}\int_{\tau_r} (f-\Pi_r^1 f)^2dx\,,  \q \forall\,\varepsilon >0      \,.
 \end{equation}
\end{theorem}
  \begin{proof}   Since   ${\mc M}_{\mr{fem}}^{(1)} (\sigma,f,{\bf z})={\mc M} (\sigma,\sigma_*,f,v,{\bf z})$ with  $\sigma_*$ defined according to  (\ref{A-Nitsche}),  theorem is a direct consequence of  Theorem~~\ref{Th: K3} and Lemma~\ref{Le: Kor}.
\end{proof}
\par The way of evaluation of the constant $c_\dag$, presented in the proof of Lemma~\ref{Le: Kor}, is rather general and can be expanded
  on the analogous a posteriori error bounds of the finite element method solutions for the  $2n^{\mr{th}}$-order elliptic equations,  $n\ge 1$, see
 Korneev   \cite{Korneev:2017-2}. The most complicated in it is evaluation of the constant $c_{\circ}$.  However, in many cases  such estimates are well known.  For instance,  if the domain is convex, then
   $
     |v|_2 \le \|\Delta v\|_0 \,,
   $
  see Ladyzhenskaya  \cite[(6.5) in ch. II]{Ladyzhenskaya:73a}.  Therefore,  at ${\bf A}={\bf I}$ and $\sigma\equiv 0$  we have $c_{\circ}\le 1$,   and  from (\ref{u-sigma-f+}) we conclude that at least $c_{\circ}\le 2$  for any $\sigma\ge 0$.  Existence of the constant $c_{\circ}$
   poses some conditions on  smoothness of the boundary and  coefficients of the equation (if they are not constant).  At the same time,
   there is the possibility to avoid the mentioned additional restrictions, except  for those related to the suitable approximation operator.  If there exists some interpolation type or other approximation operator    with locally defined approximations for functions from $H^1(\Omega)$, then it is possible to show that the {\em constants in the a posteriori bounds depend only on the local approximation properties of the finite element space}. A good example of such an operator is the
  quasi-interpolation operator of  Scott and Zhang \cite{ScottZhang:90}, which will be used below to illustrate the statement. We start from the description of the properties of this operator needed for our purpose.

\par Let ${\Omega}\subset {\mbb R}^m,\,\, m\ge2,$ be a bounded Lipschitz domain, whch is the domain of the quasiuniform triagulation ${\mc T}_h$ with vertices  $x^{(i)},\,\, i=1,2,\dots, I,$  and simplices $\tau_r$ of diameters not greater $h$.  For simplicity it is assumed that faces of simplices are plain  and that the following quasiuniformity conditions are fulfilled:
\begin{equation}\label{fedisc-reg}
      0 < {c}_\vartriangle \le {\UL{\rho}_r}/{h_r}  \,,
   	\quad  \hat{\alpha}^{(1)}h \le h_r
   	\le  h\,,
\end{equation}
 where ${\UL{\rho}_r}$ and $h_r$  are the radius of the largest inscribed  sphere and the diameter of simplex
  $\tau_{r}$, respectively.
To each vertex $x^{(i)}$, we relate $(m-1)$-dimentional simplex $\tau_i^{(m-1)}$, which is the face of one of the simplices $\tau_r$ having
$x^{(i)}$ for the vertex.  For $m$ vertices of the simplex $\tau_i^{(m-1)}$ we willl use also  notations $z_l^{(i)},\,\,
   l=1,2,\dots,m,\,$ assuming for definiteness  that $z_1^{(i)}=x^{(i)}$. Clearly  the choice of the face $\tau_i^{(m-1)}$ is not unique, but for
   $x^{(i)} \in \bo\Omega$ we always take one of the faces $\tau_i^{(m-1)} \subset \bo\Omega $. We will  formulate the result of Scott and Zhang
  using the simpler notations ${\mbb V}_\vartriangle (\Omega  )$, ${\mc V}_{\mr{tr}}(\bo \Omega)$,  and  $\7{\mbb V}_\vartriangle (\Omega  )$ for the space of continuous piece wise linear functions ${\mbb V}_{h,1}^{\,0}(\Omega)$,  its trace on the boundary, and its subspace of  functions, vanishing on the boundary,  respectively.

  \par  We define functions $\theta_i \in {\mc P}_1(\tau_i^{(m-1)})$, satisfying equations
\begin{equation}
\label{theta} 
  \int_{\tau_i^{(m-1)}} \theta_i \lambda_l^{(i)}\, dx =\delta_{1,l}\,,\quad
    l=1,2,..,m\,,
\end{equation}
 where $\lambda_l^{(i)}$ are the barycentric coordinates in $\tau_i^{(m-1)}$, corresponding to the vertices $z_l^{(i)}$,  and
   $\delta_{i,l}$ is the Kronecker's symbol.  If $\phi_i \in {\mbb V}_\vartriangle (\Omega)$ are the basis functions in $ {\mbb V}_\vartriangle (\Omega)$, defined by the equalities $\phi_i
(x_j)=\delta_{i,j},\,\, i,j=1,2,..,I,\,$ then for any  $v \in
H^1(\Omega)$ the quasi-interpolation ${\mc I}_h v $ is the function
\begin{equation}
\label{I:SZh} 
   {\mc I}_h v=\sum_{i=1}^I \left(\int_{\tau_i^{(m-1)}} \theta_i v\,dx
  \right)\phi_i(x)\,.
\end{equation}
  \begin{lemma}\label{L:ScottZhang}
    The quasi-interpolation operator  $ {\mc I}_h: H^1(\Omega)\rightarrow {\mbb V}_\vartriangle
  (\Omega)$  is a projection and has the following properties:
  \par {a)} \hspace*{3mm} ${\mc I}_h v : H^1(\Omega)\mapsto {\mbb V}_\vartriangle
  (\Omega)$ and, if  $v \in {\mbb V}_\vartriangle
  (\Omega)$, then  ${\mc I}_h v = v$,
  \par {b)}  \hspace*{3mm} $(v- {\mc I}_h v) \in \mathaccent"7017{H}\,\!\!^1(\Omega),\,$
   if  $\,v|_{\bo \Omega} \in {\mc V}_{\mr{tr}}(\bo \Omega)$,
  \par {c)} \hspace*{3.3mm} $\|v- {\mc I}_h v\|_{t,\Omega}\le c_{\mr{sz}}(t,s) h^{s-t} \|v
  \|_{s,\Omega}\,$ for  $ t=0,1,$  $s=1,2$,  and  $\forall v \in H^s(\Omega)$\,,
  \par {d)} \hspace*{3mm} $| {\mc I}_h v |_{1,\Omega}\le \breve{c}_{\mr{sz}}  |\,v\,|_{1,\Omega}\,$  и
$\,\| {\mc I}_h v \|_{1,\Omega}\le \hat{c}_{\mr{sz}}
\|\,v\,\|_{1,\Omega}\,$   $\forall v\in H^1(\Omega)$,
 \newline where
$c_{\mr{sz}}(s,t), \, \breve{c}_{\mr{sz}}$, and $\hat{c}$ are positive constants, depending on
  ${c}_\vartriangle$.
  \end{lemma}
  \begin{proof} Scott and Zhang \cite{ScottZhang:90} proved a more general result.  In a given form the lemma
    was formulated and proved by Xu and Zou  \cite{XuZou:1998a}.
  \end{proof}
     \begin{theorem}  \label{Th: Kor+1}
 Let  $\Gamma_D=\bo \Omega$,  $\psi_D \equiv 0$,
 $u\in \7{H}^1(\Omega, \Delta)$. Let also the finite element assemblage satisfies the condition
    ${\mc A}$) and generates the space $\7{\mbb V}_{h,1}^{\,0}(\Omega) \subset \7{H}^1(\Omega)$ and
${\bf z} \in {\bf H}(\Omega, {\mr{div}})$.  Then for $v=u_{\mr{fem}}$
 at any  $ \sigma\in [0,1/( c_{\mr{sz}}(0,1) h)^2]$ holds the bound
  \begin{equation}\label{K-231}
               |\!\!\,\!|\!\!\,\!| v-u|\!\,\!\!|\!\,\!\!|^2\le \Theta_{\mr{sz}}  {\mc M}_{\mr{fem}}^{(2)} (\sigma,f,{\bf z})\,,\q {\mc M}_{\mr{fem}}^{(2)} (\sigma,f,{\bf z})= {\mc M} (\sigma,\theta_{\mr{sz}}^{-1},f,u_{\mr{fem}},{\bf z})\,,
   \end{equation}
   where
 \begin{equation}\label{ident-01}
 \Theta_{\mr{sz}}=\frac{1+\wt{c}_{\mr{sz}}^2(1,1)}{1+c_{\mr{sz}}^2(0,1)h^2\sigma }\,, \q  \theta_{\mr{sz}}=c_{\mr{sz}}(0,1)^2h^2\,,
 \end{equation}
  and   $\wt{c}_{\mr{sz}}(1,1)$ is the constant, depending only upon
   ${c}_\vartriangle$ and $\hat{\alpha}^{(1)}$\!, given in  (\ref{ident-31}).
\end{theorem}
\begin{proof}
 For any $w\in \7{\mbb V}_{h,1}^{\,0}(\Omega)$ we have the equality
    \begin{equation}\label{ident-1}
\begin{array}{c}
    |\!\!\,\!|\!\!\,\!|e_{\mr{fem}} |\!\,\!\!|\!\,\!\!|^2= \int_\Omega \big[{ \nabla} (e_{\mr{fem}})\cdot \nabla (e_{\mr{fem}}) +
    \sigma e_{\mr{fem}}e_{\mr{fem}} \big]= \\  \vspace{-3mm}
    \\= \int_\Omega \big[(\nabla u_{\mr{fem}}+{\bf z})\cdot \nabla (e_{\mr{fem}}+w) -  ({\bf z}+\nabla u)\cdot
    \nabla (e_{\mr{fem}}+w)+  \\  \vspace{-3mm}
     \\ +
     \sigma (u_{\mr{fem}}-u)(e_{\mr{fem}}+w) \big]\,.
  \end{array}
\end{equation}
 Integration by parts of the second summand in square brackets of the right part and application of the Cauchy inequality
 with $\epsilon >0$  result in the inequality
    \begin{equation}\label{ident-2}
    \begin{array}{l}
|\!\!\,\!|\!\!\,\!|e_{\mr{fem}} |\!\,\!\!|\!\,\!\!|^2=\int_\Omega \big[(\nabla u_{\mr{fem}}+
{\bf z})\cdot \nabla (e_{\mr{fem}}+w)+
 \big( {\mr{div}} \,{\bf z}+\Delta u+\sigma (u_{\mr{fem}}-u)\big)(e_{\mr{fem}}+w) \big] \le
    \\  \vspace{-3mm}
    \\
   \Big\{ \|\nabla u_{\mr{fem}}+{\bf z}\|_0^2 +\frac{1}{\epsilon}\|f-\sigma u_{\mr{fem}} - {\mr{div}} \,{\bf z}\|_0^2\Big\}^{1/2}
   \Big\{ \|\nabla (e_{\mr{fem}}+w)\|_0^2
   + \epsilon\| e_{\mr{fem}}+w\|_0^2 \Big\}^{1/2}
   \end{array}
\end{equation}
\par According to Lemma~\ref{L:ScottZhang} and the definition of the operator   $Q$ of   $L_2$-projection upon   $ \7{\mbb V}_{h,1}^{\,0}(\Omega)$,
   for  $\phi -Q\phi$ with  any $\phi\in H^1(\Omega)$,   there are valid the bounds
    \begin{equation}\label{ident-3}
    \begin{array}{l}
      \|\phi -Q\phi\|_0 \le \|\phi \|_0\,,   \\  \vspace{-3mm}
    \\
     \|\phi -Q\phi\|_0 \le c_{\mr{sz}}(0,1) h \|\nabla \phi \|_0\,,   \\  \vspace{-3mm}
    \\
     \|\nabla(\phi -Q\phi)\|_0 \le \wt{c}_{\mr{sz}}(1,1)  h \|\nabla \phi \|_0 \,,
     \end{array}
\end{equation}
   in which the  constant $\wt{c}_{\mr{sz}}(1,1) $ depends only on ${c}_\vartriangle$ and $\hat{\alpha}^{(1)}$.
   The proof is needed only  for  the last bound, and it  follows from the relations
  $$
  \begin{array}{c}
  \|\nabla(\phi -Q\phi)\|_0 \le \|\nabla(\phi -{\mc I}_h\phi)\|_0  +\|\nabla ({\mc I}_h\phi -Q\phi)\|_0\le \breve{c}_{\mr{sz}}\|\nabla \phi \|_0 +\\
 \vspace{-3mm}
    \\
  c_{1,0}h^{-1}\|{\mc I}_h\phi -Q\phi\|_0 \le \breve{c}_{\mr{sz}}\|\nabla \phi \|_0 + c_{1,0}h^{-1} \Big[ \|{\mc I}_h\phi -\phi \|_0
  +\|\phi  -Q\phi\|_0 \Big] \le
  \\
 \vspace{-3mm}
    \\
    \Big(\breve{c}_{\mr{sz}}+ 2c_{1,0}  c_{\mr{sz}}(0,1)    \Big)\|\nabla \phi \|_0 \,,
  \end{array}
  $$
  where $c_{1,0}$ is the constant in the inverse inequality
$$
\|\nabla ({\mc I}_h\phi -Q\phi)\|_0\le c_{1,0}h^{-1}\|{\mc I}_h\phi -Q\phi\|_0\,.
$$
Therefore,
  \begin{equation}\label{ident-31}
                      \wt{c}_{\mr{sz}}(1,1) =\breve{c}_{\mr{sz}}+ 2c_{1,0}  c_{\mr{sz}}(0,1) \,.
  \end{equation}
  \par It is worth noting,  that the third inequality (\ref{ident-3}), indicating the stability in $H^1(\Omega)$ of  $L_2$-projection,  was proved
   by Bramble and Xu  \cite{BrambleXu:91} in a different way with a different way of evaluationg the constant $\wt{c}_{\mr{sz}}(1,1)$.

  \par For the reason that  $w=Qe_{\mr{fem}}\in \7{\mbb V}_{h,1}^{\,0}(\Omega) $, it can be adopted  $w=Qe_{\mr{fem}}$.  Combining with
    (\ref{ident-3})  and setting   $\epsilon=\sigma_{\mr{sz}}:= (c_{\mr{sz}}(0,1) h)^{-2}$ leads the bound
    \begin{equation}   \label{ident-4}
    \begin{array}{c}
     \|\nabla (e_{\mr{fem}}+w)\|_0^2
   + \sigma_{\mr{sz}}\| e_{\mr{fem}}+w\|_0^2=\|\nabla (e_{\mr{fem}}+w)\|_0^2+\beta\sigma \| e_{\mr{fem}}+w  \|_0^2 +
   \\  \vspace{-3mm}
    \\
     +(\sigma_{\mr{sz}}-\beta\sigma) \| e_{\mr{fem}}+w\|_0^2 \le  \wt{c}_{\mr{sz}}^2(1,1)\|\nabla e_{\mr{fem}}\|_0^2 +\beta\sigma \| e_{\mr{fem}}  \|_0^2
     +  \frac{\sigma_{\mr{sz}} -\beta\sigma}{\sigma_{\mr{sz}}} \|\nabla e_{\mr{fem}}\|_0^2  \,.
   \end{array}
\end{equation}
   On the basis of (\ref{ident-4})  we conclude that
    \begin{equation}\label{ident-5}
    \|\nabla (e_{\mr{fem)}}+w\|_0^2
   + \sigma_{\mr{sz}} \| e_{\mr{fem}}+w\|_0^2\le \frac{1+  \wt{c}_{\mr{sz}}^2(1,1)}{1+ \kappa  }
   \big[\,\|\nabla e_{\mr{fem}}\|_0^2+
   \sigma \| e_{\mr{fem}}  \|_0^2 \, \big]
    \end{equation}
 with $\kappa=\sigma/\sigma_{\mr{sz}}$.  Now from (\ref{ident-2}) and (\ref{ident-5}) the theorem follows .
\end{proof}
\begin{remark}\label{Re-11}
   Quasi-interpolation operator $ {\mc I}_h$  is defined in \cite{XuZou:1998a} on triangultions, satisfying the conditions of the shape regularity
   $$
    0 < {c}_\vartriangle \le{\UL{\rho}_r}/{h_r}, \,\q h_r\le  h
    $$
    with preserving the properties  a), b) and the properties с), d) taking the form
\begin{equation}\label{SZh-geo}
\begin{array}{l}

\|v- {\mc I}_h v\|_{t,\tau_r}\le c_{\mr{sz}}(t,s) h_r^{s-t} \|v
    \|_{s,\delta_r}\,,\q   t=0,1, \,\, s=1,2,\, \, \,\,\,  \forall v\in H^1(\delta_r)\,,
  \\  \vspace{-3mm}
    \\
  | {\mc I}_h v |_{1,\tau_r}\le \breve{c}_{\mr{sz}}  |\,v\,|_{1,\delta_r} \,\,\, \mbox{and}\,\,\,
    \| {\mc I}_h v \|_{1,\tau_r}\le \hat{c}_{\mr{sz}} \|\,v\,\|_{1,\delta_r}\,\,\,  \forall v\in H^1(\delta_r) \,,
    \end{array}
  \end{equation}
  where $c_{\mr{sz}}(t,s)={\mr{const}}$,  ${\delta}_r={\mr{interior}}\{ \cup_\varkappa \OL{\tau}_\varkappa:
      \OL{\tau}_\varkappa\cap
      \OL{\tau}_r\neq \varnothing\}$  and  $r=1,2,\dots,{\mc R}$. More over, these authors designed also the quasi-interpolation operators
        $ {\mc I}_{h,p}: H^1(\Omega)\rightarrow {\mbb V}_{h,p}(\Omega)$, for which again the properties a), b) are preserved, but in (\ref{SZh-geo})
        $s=1,2,\dots,p+1$.  These facts lead to some important generalizations  of Theorem~\ref{Th: Kor+1}. One of them is the expansion
        of the a posteriori error bounds  (\ref{K-231})-(\ref{ident-01}) on the solutions by the finite element methods from the spaces ${\mbb V }_{h,p}(\Omega), \,\, p>1$. In this case, the constants related to $ {\mc I}_h$  must be replaced by the respective constants  related to  the operator $ {\mc I}_{h,p}$.  Other  possibility  is the generalization on the case of piece wise constant  reaction coefficient $\sigma\big|_{\tau_r}\!=\!\sigma_r\!=\!{\mr{const}}\!\ge\!0$.
   \end{remark}

  \section{\normalsize  Consistency with a priori bounds, inverse like bound
      } \label{Se:cons-inverse}  

\par  This Section is dedicated to the properties of the suggested a posteriori error majorants,
which are important for the construction of the convergent adaptive algorithms.
   \vspace{0.2cm}

\par  {\bf \em 4.1.  Consistency}
    \par   More accurate a posteriori  majorants (\ref{K-fem+1}), (\ref{K-fem+2}) and (\ref{ident-01}) of the finite element methods errors
    are consistent with the unimprovable a priori error estimates.  In order to become convinced in this, let us first present the a priori approximation error estimates.
    \par \par  If $v\in H^l(\Omega)$  and its finite element approximation $\wt{v}$ belongs to   ${\mbb V}_h(\Omega)={\mbb V}_{h,p}^0(\Omega), \,\,1\le l\le p+1,$  then there are several ways to find such $\wt{v}$ that
 \begin{equation}\label{approx+}
             \|v - \wt{v}
             \|_{k,\Omega} \le c_{k,l}h^{l-k} \| v\|_{l,\Omega}\,,  \q   k=0,1\,,\quad c_{k,l}={\mr{const}}\,,
\end{equation}
  where $c_{k,l}=c_{k,l}(k,l, c_\vartriangle, \hat{\alpha}^{(1)})$, if the condition   ${\mc A}$)  is fulfilled. In a more general case of the curvilinear finite elements  $ c_\vartriangle, \hat{\alpha}^{(1)}$  can be understood as the corresponding characteristics  from the conditions of generalized quasiuniformity.  The bounds (\ref{approx+}) can be found, {\em e.g.}, in \cite{Ciarlet:1978a, Korneev:77, ScottZhang:90}. At $p=1$, $l=1,2$, and $\wt{u}={\mc I}_hu$  they are a concequence, for instance, of Lemma~\ref{L:ScottZhang}. In general,  if $v=u$ is the solution of the reaction-diffusion problem, then
    for  $\wt{v}$, as it follows from  (\ref{L_2_K}), one can  take  any of functions  $u_\circ,\,  u_{\mr{fe}}$, or  $u_s$.
\begin{lemma}  \label{Le: Kor3}
    Let $\Gamma_D=\bo \Omega$, $\sigma\equiv {\mr{const}}\ge 0$, $\psi_D \equiv 0$,   $u\in \7{H}^1(\Omega)\cap H^l(\Omega)$. Let also
   finite element assemblage
  generates the space ${\mbb V}_{h,p}^0(\Omega),\,\,p\ge 1$. Then for $1\le l\le p+1$ there are hold the bounds
    \begin{eqnarray}
        \| e_{\mr{fem}} \|_{0} &\le & c_{0,l} h^{l}\|u
        \|_{l}\,,    \label{fem-conv-0}
       \vspace*{-4mm}
                 \\
          \| e_{\mr{fem}} \|_{\bf A} & \le & \sqrt{\mu_2}c_{1,l} h^{l-1}\|u
            \|_{l}\,.  \label{fem-conv-A}
     \end{eqnarray}
      with the constants  $c_{k,l}$ from (\ref{approx+}).
      Besides, for $\sigma \le c_\dag^2 h^{-2}$  and for  $\sigma \ge c_\dag^2 h^{-2},\,\, f\in L_2(\Omega)$, respectively we have the bounds
    \begin{eqnarray}
  |\!\!\,\!|\!\!\,\!| e_{\mr{fem}}|\!\,\!\!|\!\,\!\!|^2   & \le & (c_{0,l}^2+\mu_2c_{1,l}^2 ) h^{2(l-1)}\|u
            \|_{l}^2\,,    \label{fem-conv-e1}
            \vspace*{-4mm}
                 \\
                 |\!\!\,\!|\!\!\,\!| e_{\mr{fem}}|\!\,\!\!|\!\,\!\!|^2   & \le & \mu_2c_{1,l} ^2 h^{2(l-1)}\|u
            \|_{l}^2+ \sigma^{-1} \|f\|_0^2 \,. \label{fem-conv-e2}
 \end{eqnarray}
    \end{lemma}
    \begin{proof} Lemma follows from  (\ref{L_2_K}),  (\ref{approx+}), (\ref{u-sigma-f})  and Lemms~\ref{Le: Kor} and  \ref{L:ScottZhang}.
  \end{proof}
  \par  Consistency of the a posteriori  error bounds (\ref{K-fem+1}), (\ref{K-fem+2}) and (\ref{K-231})-(\ref{ident-01}) with the a priori bound (\ref{fem-conv-e1}) is established in the same way, and for the finite element solutions of an increased smootheness $u_{\mr{fem}}\in {\mbb V}_{h,p}^1(\Omega)\subset H^2(\Omega)$ it is practically evident. Note, that for such solutions additionally
 \begin{equation}\label{approx+2l}
             \|u - u_{\mr{fem}}
             \|_{2} \le \hat{c}_{2,l}h^{l-2} \| u\|_{l,\Omega}\,, \q l\ge 2, \quad \hat{c}_{2,l}={\mr{const}}\,.
\end{equation}
   For the proof it is sufficient to implement the inverse inequality, the inequalities (\ref{A}), the second inequality (\ref{L_2_K})
   and approximation estimates  (\ref{approx+}):
   $$
\begin{array}{c}
\| e_{\mr{fem}} \|_{2} \le \|u - u_{\mr{int}} \|_{2}+\|u_{\mr{int}} - u_{\mr{fem}} \|_{2}\le \|u - u_{\mr{int}} \|_{2}+
ch^{-1}\|u_{\mr{int}} - u_{\mr{fem}} \|_{1}\le \\
\vspace*{-3mm}
                 \\
   \le  \|e_{\mr{int}} \|_{2}+ c_{\mr{i}}h^{-1}(\|e_{\mr{int}}  \|_{1}+ \sqrt{\frac{1}{\mu_1}}
\|e_\circ \|_{\bf{A}}) \le  \\
\vspace*{-3mm}
                 \\  \le h^{l-2}\big[c_{2,l}+c_{\mr{i}}c_{1,l}(1 +\sqrt{\frac{\mu_2}{\mu_1}} ) \big] \| u\|_{l,\Omega} =\hat{c}_{2,l}h^{l-2} \| u\|_{l,\Omega}\,.
\end{array}
$$
\par  Let us introduce for the right part of (\ref{K-fem+1}) the notation
      $$
      \eta^2_1(e_{\mr{fem}})  =\frac{2}{1+{\mr{c}}_\dag^{2} h^2\sigma  }{\mc M}_{\mr{fem}}^{(1)} (\sigma,{f},{\bf z})\,.
    $$
  In case $u_{\mr{fem}}\in {\mbb V}_{h,p}^1(\Omega)$ it can be adopted ${\bf z}={\bf z}_{\mr{fem}}=-{\bf A}\nabla {\bf z}_{\mr{fem}}$,
  and the diffusion component in $\eta_1(e_{\mr{fem}}) $ vanishes, {\em i.e.},  $ ]\!| {\bf A}\nabla\,u_{\mr{fem}} +{\bf z}_{\mr{fem}} |\![_{{\bf A}^{-1}}=0$. Therefore, taking into account the inequality $\sigma\le 1/({\mr{c}}_{\dag}h)^2$, the first inequality  (\ref{L_2_K}), and the estimates   (\ref{fem-conv-0}),  (\ref{approx+2l}),  we will have
 \begin{equation}\label{consist}
 \begin{array}{c}
  \eta^{(1)}(e_{\mr{fem}})
 \le \sqrt{2}{\mr{c}}_{\dag}h\|{f}-\sigma u_{\mr{fem}}  - {\mr{div}} \,{\bf z}\|_{0}\le  \sqrt{2} {\mr{c}}_{\dag} h
 \| {\mathfrak L}e_{\mr{fem}}\|_0  \le
 \\
  \vspace*{-3mm}
               \\
     \le
 \sqrt{2} \Big[\frac{1}{ {\mr{c}}_{\dag}h} \|e_{\mr{fe}}\|_0+{\mr{c}}_{\dag}\mu_2 h | e_{\mr{fem}}|_2 \Big]\le
  \sqrt{2}  h^{l-1} \Big[ \frac{c_{0,l}}{ {\mr{c}}_{\dag}}
   +
  {\mr{c}}_{\dag}  c_{2,l}\mu_2\Big]\|u\|_l\,.
 \end{array}
\end{equation}
   In the order of $h$   this bound is the same as the unimprovable in the order a priori bound   (\ref{fem-conv-e1}).

\par   Suppose $u_{\mr{fem}}\in {\mbb V}_{h,p}(\Omega)\subset H^1(\Omega),\,\, {\mbb V}_{h,p}(\Omega)\nsubseteq H^2(\Omega)$. Then it is natural to require that for the recovered flax ${\bf z}$ the same in the order estimates of convergence, as the ones reflected for the flax ${\bf z}_{\mr{fem}}$ in Lemma~\ref{Le: Kor3},  were hold
alongside with the  estimate
$$
\| \mr{div} (-{\bf A}\nabla u- {\bf z})
             \|_{0} \le \tilde{c}_{2,l}h^{l-2} \| u\|_{l,\Omega},\,\, \tilde{c}_{2,l}={\mr{const}}\,,
             $$
corresponding to  (\ref{approx+2l}).  Note, that the latter estimate follows from the former in the same way  as (\ref{approx+2l}) follows from
(\ref{approx+}).  If the pointed out requirements are fulfilled the bound
     \begin{equation}\label{consist+1}
     \eta_1(e_{\mr{fem}})
     \le c_1 h^{l-1} \|u\|_l, \,\, c_1={\mr{const}}\,,
     \end{equation}%
  is proved  similarly to the similar bound  (\ref{consist}).
      \par It is worth noting that the a posteriori bounds, derived in this work, differ from a number of known bounds only in the coefficients
      before the norms in their right parts. Our coefficients have smaller or the same as earlier known orders.   For this reason the efficient flax recovery algorithms, suggested earlier,  are efficient  for our bounds
      as well.

      \par  As was mentioned above, an a posteriori bound is  unimprovable in the  order, if it is  consistent with the a priori bound unimprovable in the  order.  This means that in the class of solutions, for which the a priori bound is unimprovable in the  order,
      there exist such that
     \begin{equation}\label{consist+2}
                 \eta_1(e_{\mr{fem}}) \le  {\mbb C}_\star |\!\!\,\!|\!\!\,\!|\, e_{\mr{fem}}\,|\!\,\!\!|\!\,\!\!|\,.
    \end{equation}%
Indeed, let $f\in L_2(\Omega)$, $d=2$, and the inequality (\ref{2-f}) is fulfilled. Then the inequalities  (\ref{fem-conv-e1}) с $l=2$ and
 (\ref{consist+1})  are fulfilled as well.  At the same time,  such $f\in L_2(\Omega)$ exists that
     \begin{equation}\label{consist+3}
     \begin{array}{c}
   h\|u\|_2 \le c_2\,  \inf_{\phi\in {\mbb V}_h(\Omega)} \|\nabla(u-\phi)\|_0=c_2\, \|\nabla e_{\mr{fe}}\|_0\le
    \\
  \vspace*{-3mm}
               \\
    c_2\, \|\nabla e_{\mr{fem}}\|_0\le c_2\, |\!\!\,\!|\!\!\,\!|\, \nabla  e_{\mr{fem}}\,|\!\,\!\!|\!\,\!\!|\,,
     \q c_2={\mr{const}}\,.
   \end{array}
\end{equation}%
     The first of these inequalities follows from the estimates of the $N$-width of the compuct of functions $v\in \7{H}^1(\Omega), \,\, \|v\|_1=1$,  for $N=h^{-2}$, see  \cite[Гл. 4, п. 4.1]{OganesianRuhovets:1979}.  In turn, from  (\ref{consist+1}), (\ref{consist+3}) follows (\ref{consist+2}),  which together with (\ref{K-fem+1}) yield the two sided bound
     \begin{equation}\label{consist+4}
         |\!\!\,\!|\!\!\,\!|\, e_{\mr{fem}}\,|\!\,\!\!|\!\,\!\!|^2\le
                 \eta_1(e_{\mr{fem}}) \le {\mbb C}_\star |\!\!\,\!|\!\!\,\!|\, e_{\mr{fem}}\,|\!\,\!\!|\!\,\!\!|^2\,.
    \end{equation}%

\vspace{0.2cm}

\par  {\bf \em 4.2.  Inverse like inequality}
\par Derivation of a posteriori error majorants for numerical solutions was often accompanied by creation of the flax recovery alorithms,
including algorithms with the equilibration,  and by the proofs with their help of the inverse like and the local effectiveness bounds, see \cite{ErnStephansenVohralik:2009, Cai and S. Zhang :2010,
 CreuseNicaise:2010, AinsworthVejchodsky:2011, AinsworthVejchodsky:2015,  CareyCarey:2011, CarstensenMerdon:2013}.  As was already mentioned, some of these results, accompanying  majorants obtained by other authors, are strightforwardly  expandable on the  majorants, suggested in this paper.  More over,  the range of admissible flaxes, for which, {\em e.g.}, the inverse like bounds  hold, can be widened, because the equilibration of flaxes can became unnecessary due to a more adequate representation of the residual terms in the majorants. We illustrate these inferences by one example.
 \begin{theorem} \label{Th:45} 
  Let the conditions of Theorem~\ref{Th: Kor}  are fulfilled  and ${\bf A}={\bf I}$, $u_{\mr{fem}} \in {\mbb V }_{h,1}(\Omega)$. Let  also  the  vector-function  ${\bf z}\in {\mbf W}_{h,2}(\Omega, {\mr{div}})$
    be defined as  ${\bf L}^2$-projection of  the  vector-function    ${\bf z}_{\mr{fem}}$ on  the space
   ${\mbf W}_{h,2}(\Omega, {\mr{div}})$.  Then for $k=1,2$
  \begin{equation}\label{002}
       {\mc M}_{\mr{fem}}^{(k)}(\sigma,f,{\bf z})\le {\mbb C} \big[\,|\!\!\,\!|\!\!\,\!| e_{\mr{fem}}|\!\,\!\!|\!\,\!\!|^2+
       \sum_{\varkappa=1}^{\mc R} \frac{h_\varkappa^2}{\pi^2}\int_{\tau_\varkappa} (f-\Pi_\varkappa^1 f)^2dx \big]\,  
\end{equation}
  with the constant ${\mbb C}={\mbb C}(\Omega,{c}_\vartriangle)$.
  \end{theorem}
  \begin{proof}  Let us consider the majorant  ${\mc M^{(1)}}$.  Suppose,
  there exists ${\bf y}_\circ \in
   {\mbf W}_{h,\kappa}(\Omega, {\mr{div}}), \,\,\kappa\ge 1$,  for which the estimate  (\ref{002}) holds with $ {\bf z}={\bf y}_\circ $,  {\em i.e.},
  \begin{equation}\label{0021}
       {\mc M}_{\mr{fem}}^{(k)}(\sigma,f,{\bf y}_\circ)\le {\mbb C}_\circ \big[\,|\!\!\,\!|\!\!\,\!| e_{\mr{fem}}|\!\,\!\!|\!\,\!\!|^2+
       \sum_r \frac{h_r^2}{\pi^2}\int_{\tau_r} (f-\Pi_r^1 f)^2dx \big]\,.  
\end{equation}
  Let also ${\bf y} $ be orthogonal ${\bf L}^2$-projection of vector function  ${\bf z}_{\mr{fem}}$ upon  the space $
   {\mbf W}_{h,\kappa}(\Omega, {\mr{div}})$. If to use the inequality
     \begin{equation}\label{003}
        \|\nabla u_{\mr{fem}}-  {\bf y} \|^2_{{\bf L}^2(\Omega)}\le \|\nabla u_{\mr{fem}}-  {\bf y}_\circ \|^2_{{\bf L}^2(\Omega)}\,,
     \end{equation}
    and then  the inequality   (\ref{0021}),   we obtain
  \begin{equation}\label{004}
  \begin{array}{c}
       {\mc M}_{\mr{fem}}^{(1)}(\sigma,f,{\bf y})  \le 2 \Big\{ {\mc M}_{\mr{fem}}^{(1)}(\sigma,f,{\bf y}_\circ)  +
        ({\mr{c}}_\dag h)^2\|{\mr{div}} ({\bf y}_\circ- {\bf y})\|_{L_2(\Omega)}^2 \Big\}  \le  \\
 \vspace{-3mm}
 \\
     2\Big\{ {\mbb C}_\circ \big[\,|\!\!\,\!|\!\!\,\!| e_{\mr{fem}}|\!\,\!\!|\!\,\!\!|^2+
    \sum_r \frac{h_r^2}{\pi^2}\int_{\tau_r} (f-\Pi_r^1 f)^2dx \big] +
     ({\mr{c}}_\dag h)^2 \|{\mr{div}} ({\bf y}_\circ- {\bf y})\|_{L_2(\Omega)}^2 \Big\} \,
 \end{array}
 \end{equation}
   The difference ${\bf y}_\circ- {\bf y}$ belongs to the finite element space  ${\mbf W}_{h,\kappa}(\Omega, {\mr{div}})$ and, therefore,
    the inverse inequality is fulfilled. The use of the inverse inequality  and  (\ref{003}), (\ref{0021}) yields the bound
  \begin{equation}\label{005}
  \begin{array}{c}
   {\mr{c}}_\dag h \|{\mr{div}} ({\bf y}_\circ- {\bf y})\|_{L_2(\Omega)} \le {\mr{c}}_{\mr{i}}{\mr{c}}_\dag \| {\bf y}_\circ- {\bf y}\|_{L_2(\Omega)} \le {\mr{c}}_{\mr{i}}{\mr{c}}_\dag \big[\,\| {\bf y}_\circ- \nabla u_{\mr{fem}} \|_{L_2(\Omega)}  +
      \\
  \vspace{-3mm}
   \\
     +  \| \nabla u_{\mr{fem}} - {\bf y}\|_{L_2(\Omega)}  \big] \le  2{\mr{c}}_{\mr{i}}{\mr{c}}_\dag \big[\,\| {\bf y}_\circ- \nabla u_{\mr{fem}} \|_{L_2(\Omega)}\le 2{\mr{c}}_{\mr{i}}{\mr{c}}_\dag {\mbb C}_\circ^{1/2} \,|\!\!\,\!|\!\!\,\!| e_{\mr{fem}}|\!\,\!\!|\!\,\!\!|\,,
 \end{array}
 \end{equation}
  which together with(\ref{004}) means validity of  (\ref{002}) for $ {\bf z}={\bf y}$ under the condition of  existence of  ${\bf y}_\circ $
   with the pointed out property.
    \par To complete the proof for $k=1$ and $\kappa=2$ we note that Ainsworth  and Vejchodsky  \cite{AinsworthVejchodsky:2015}  suggested the algorithm for  evaluation  of such a flax
     ${\bf z}_{_{AV}}\in {\mbf W}_{h,2}(\Omega, {\mr{div}})$, that for ${\bf y}_\circ={\bf z}_{_{AV}}$ the inequality (\ref{0021}), obviously, holds.  The proof of the bound  (\ref{002})  for $k=2$ is only slightly different.
   \end{proof}

   \par Let $u$ and  $u_{\mr{fem}}$ are the exact and the finite element solutions to the problem in Theorem~\ref{Th:45} and $\hat{u}$ and
        $\hat{u}_{\mr{fem}}$ be the respective solutions of the same problem, but with the right part $\hat{f}$ instead of $f$. It is easy to notice that
        $\hat{u}_{\mr{fem}}=u_{\mr{fem}}$.  Therefore, as a consequence of (\ref{K-fem+1}) and (\ref{002}) for $\hat{e}_{\mr{fem}}=\hat{u}_{\mr{fem}}-\hat{u}$,  we have
        $$
                  |\!\!\,\!|\!\!\,\!|\hat{ e}_{\mr{fem}}|\!\,\!\!|\!\,\!\!|^2\le \frac{2}{1+{\mr{c}}_\dag^{2} h^2\sigma  } \, {\mc M}_{\mr{fem}}^{(1)}(\sigma,\hat{f},{\bf z})     \le  {\mbb C} \,  |\!\!\,\!|\!\!\,\!| \hat{ e}_{\mr{fem}}|\!\,\!\!|\!\,\!\!|^2\,,
        $$
  with the same $\bf z$ as in Theorem~\ref{Th:45}.  Similar inequalities hold  for  $ {\mc M}_{\mr{fem}}^{(2)}$.

\vspace{0.2cm}

\par  {\bf Acknowledgements}

   \par  The author expresses his sincere gratitude to doctor V. S.
          Kostylev
          for helpful discussions, instrumental for some results.

\end{document}